\documentclass[12pt]{amsart}
\usepackage{fullpage,url,amssymb,amsxtra,enumerate,colonequals, multirow}
\usepackage[all]{xy} 
\usepackage{mathrsfs} 
\usepackage{array, booktabs, tikz-cd}
\usepackage[T1]{fontenc}

\usepackage{color}

\newcommand{\defi}[1]{\textsf{#1}} 

\newenvironment{enumalph}
{\begin{enumerate}}
{\end{enumerate}}

\newenvironment{enumroman}
{\begin{enumerate}}
{\end{enumerate}}

\newenvironment{enumromanii}
{\begin{enumerate}}
{\end{enumerate}}

\newcommand{\NE}{N_E}

\newcommand{\C}{\mathbb{C}}
\newcommand{\F}{\mathbb{F}} 
\newcommand{\Fbar}{F^{\textup{al}}}
\newcommand{\Fun}{F^{\textup{un}}}
\newcommand{\Qun}{\Q^{\textup{un}}}

\newcommand{\Q}{\mathbb{Q}}
\newcommand{\R}{\mathbb{R}}
\newcommand{\Z}{\mathbb{Z}}

\newcommand{\rhobar}{{\overline{\rho}}}

\newcommand{\tauex}[1]{\tau_{\textup{ex},2,#1}}

\newcommand{\tauspnew}[1]{\tau_{\textup{St},#1}}
\newcommand{\tauPSnew}[2]{\tau_{\textup{ps},#1}{(#2)}}
\newcommand{\tauSCnew}[2]{\tau_{\textup{sc},#1}{(#2)}}

\newcommand{\vv}{v}

\newcommand{\eps}{\varepsilon}

\newcommand{\epschar}[1]{\varepsilon_{#1}}
\newcommand{\chichar}[1]{\chi_{#1}}

\newcommand{\spArt}{^{\textup{A}}}
\newcommand{\spart}{\spArt}

\newcommand{\spun}{{}^{\textup{un}}}

\newcommand{\Mquad}{K}

\newcommand{\calO}{\mathcal{O}}

\newcommand{\fp}{\mathfrak{p}}

\newcommand{\frakf}{\mathfrak{f}}

\newcommand{\frakp}{\mathfrak{p}}
\newcommand{\frakq}{\mathfrak{q}}

\DeclareMathOperator{\condexp}{condexp}

\DeclareMathOperator{\Fr}{Fr}
\DeclareMathOperator{\cond}{cond}
\DeclareMathOperator{\Art}{Art}

\DeclareMathOperator{\Gal}{Gal}

\DeclareMathOperator{\Ind}{Ind}

\DeclareMathOperator{\Nm}{Nm}

\DeclareMathOperator{\ord}{ord}

\DeclareMathOperator{\Sp}{St}

\DeclareMathOperator{\PS}{PS}

\DeclareMathOperator{\ver}{ver}
\DeclareMathOperator{\WD}{WD}

\newcommand{\BCKF}{\Ind_{W_{\Mquad}}^{W_F}}

\newcommand{\GL}{\operatorname{GL}}

\newcommand{\PGL}{\operatorname{PGL}}

\newcommand{\SL}{\operatorname{SL}}

\usepackage[
        colorlinks, 
        citecolor=blue,
        backref,
        urlcolor=blue,
        linkcolor=blue,
        pdfauthor={Lassina Demb\'el\'e, Nuno Freitas, and John Voight}, 
        pdftitle={On Galois inertial types of elliptic curves}
]{hyperref}

\newcommand{\LMFDBE}[1]{\href{https://www.lmfdb.org/EllipticCurve/Q/#1}{\textsf{#1}}}
\newcommand{\LMFDBL}[1]{\href{https://www.lmfdb.org/padicField/#1}{\textsf{#1}}}
\newcommand{\LMFDBG}[1]{\href{https://www.lmfdb.org/GaloisGroup/#1}{\textsf{#1}}}

\numberwithin{equation}{subsection}
\newtheorem{theorem}[equation]{Theorem}
\newtheorem{lemma}[equation]{Lemma}
\newtheorem{corollary}[equation]{Corollary}
\newtheorem{cor}[equation]{Corollary}
\newtheorem{proposition}[equation]{Proposition}
\newtheorem*{maintheorem*}{Main Theorem}

\theoremstyle{definition}
\newtheorem{definition}[equation]{Definition}

\theoremstyle{remark}
\newtheorem{example}[equation]{Example}

\newtheorem{remark}[equation]{Remark}

\usepackage{comment, todonotes}

\begin{document}

\title{On Galois inertial types of elliptic curves over $\Q_p$}

\author{Lassina Demb\'el\'e}
\address{Department of Mathematics, King's College London, Strand, London WC2R 2LS, UK}
\email{lassina.dembele@kcl.ac.uk}

\author{Nuno Freitas}
\address{
Instituto de Ciencias Matem\'aticas, CSIC, 
Calle Nicol\'as Cabrera
13--15, 28049 Madrid, Spain}
\email{nuno.freitas@icmat.es}

\author{John Voight}
\address{Department of Mathematics, Dartmouth College, 6188 Kemeny Hall, Hanover, NH 03755, USA}
\email{jvoight@gmail.com}
\urladdr{\url{http://www.math.dartmouth.edu/~jvoight/}}

\date{\today}

\begin{abstract}
We provide a complete, explicit description of the inertial Weil--Deligne types arising from elliptic curves over $\Q_p$ for $p$ prime.  
\end{abstract}

\maketitle

\section{Introduction}

\subsection{Motivation}

In the classification of finite-dimensional complex representations of the absolute Galois group of a local field, it has proven to be very useful to classify by restriction to the inertia subgroup \cite{bk93,Pas05,bh14,lat15}.  In this article, we will pursue an explicit classification for such representations coming from elliptic curves.

Let $p$ be prime and let $F \supseteq \Q_p$ be a finite extension with algebraic closure $\Fbar$.  Let $W_F$ be the Weil group of $F$, the subgroup of $\Gal(\Fbar\,|\,F)$ acting by an integer power of the Frobenius map on the maximal unramified subextension.  Let $(\rho\colon W_F \to \GL_n(\C),N)$ be an $n$-dimensional (complex) Weil--Deligne representation 
(Definition \ref{defn:WDRep}).  Let $I_F \leq W_F$ be the inertia subgroup.  An \defi{inertial (Weil--Deligne) type} (also called a \defi{Galois inertial type}) is a pair $(\tau,N)$ where $\tau = \rho|_{I_F}$ for a Weil--Deligne representation $(\rho,N)$.  To ease notation, we will often abbreviate the pair $(\tau,N)$ by $\tau$ (and indeed often we have $N=0$ anyway).  

Already the case $n=2$ is interesting and rich, and we will consider this case here.  Inertial types for $2$-dimensional representations were introduced by Conrad--Diamond--Taylor \cite{cdt99} and Breuil--Conrad--Diamond--Taylor \cite{bcdt01} in the study of deformation rings of Galois representations and were used in the proof of modularity of elliptic curves over~$\Q$.  
Diamond--Kramer \cite[Appendix]{dia95} described the analogously defined type (as in Diamond \cite{Diamond})
of the mod $p$ Galois representation~$\overline{\rho}_{E,p}$ attached to an elliptic curve $E$ over $F$ in terms of the $j$-invariant of $E$; in particular, they give
a description of the restriction of $\overline{\rho}_{E,p}$ to $I_F$ in as much detail as possible using only $j(E)$.

Types have also been studied in the context of Galois representations attached more generally to classical modular forms. For example, in Loeffler--Weinstein~\cite{LW2012, LW2015} 
an algorithm to determine the restriction to decomposition groups of such representations was described and implemented; this includes a description of the inertial type.
By counting the number of inertial types attached to modular forms, Dieulefait--Pacetti--Tsaknias~\cite{DPT} have given a precise generalization of the Maeda conjecture.
 
Additionally, inertial types have played a prominent role in the mod $p$ and $p$-adic Langlands program.  Henniart \cite[Appendix]{bm02} showed that there is an \emph{inertial Langlands correspondence} between $2$-dimensional Galois inertial types of $F$ and smooth representations of $\GL_2(\calO_F)$,
where $\calO_F$ denotes the ring of integers of $F$.  Indeed, the Breuil--M\'ezard conjecture~\cite{bm02} for $\Q_p$ can be seen as a refinement of Serre's conjecture over $\Q$, where inertial types are a crucial input.  An inertial Langlands correspondence for general $n \geq 2$ was proven by Pa\v{s}k\={u}nas \cite{Pas05}.  

Diophantine applications provide another important motivation to study inertial types for $\GL_2$. In Bennett--Skinner~\cite{BenSki}, the \emph{image of inertia argument} was introduced and successfully applied to solve certain Fermat equations. Recently, further refinements and applications of this argument were obtained 
by Billerey--Chen--Dieulefait--Freitas \cite{BCDFmulti}: we may be able to distinguish between the mod~$p$ representations attached to elliptic curves over a global field by showing they have different images of inertia \cite[Section~3]{BCDFmulti}.
Therefore, the more we know about inertial types of elliptic curves, the greater the applicability of this argument. In this direction, 
Freitas--Naskr{\k e}cki--Stoll \cite[Theorem~3.1]{FNS23n} describe the possible fixed fields of the restriction $\overline{\rho}_{E,p}|_{I_{\Q_p}}$ to inertia for elliptic curves $E$ over $\Q_p$ with certain reduction types at $p=2,3$, and they applied this to study solutions of the generalized Fermat equation $x^2+y^3=z^p$.

In light of these applications, the goal of this paper is to give a complete, explicit description of the inertial types for all elliptic curves $E$ over $\Q_p$.
Our main theorem (Theorem \ref{thm:mainthm} below) has already been applied to the determination of the symplectic type of isomorphisms between the $p$-torsion of elliptic curve by Freitas--Kraus \cite{fk17}.

\subsection{Main result}

Let $E$ be an elliptic curve over $F=\Q_p$.  Attached to $E$ is an inertial Weil--Deligne type $\tau_E$ obtained from the action on the (dual of the) $\ell$-adic Tate module for a prime $\ell \neq p$, independent of $\ell$ (for details, see section \ref{sec:3point1}).  If $E$ has potentially good reduction, then this good reduction is obtained over a minimal finite extension $L \supseteq F\spun$ where $F\spun$ denotes the maximal unramified extension of~$F$, and we define the \defi{semistability defect} of $E$ to be $e_E \colonequals [L:F\spun]$.

Our main result (combining Lemma~\ref{L:e=2}, Propositions \ref{prop:Emult}, \ref{P:peq5}, and \ref{P:elleq3prop}, and Theorems \ref{thm:nonexceptell2} and \ref{thm:exceptell2}) is as follows.   

\begin{maintheorem*} \label{thm:mainthm}
Let $E$ be an elliptic curve over $\Q_p$ with conductor $\NE$ and inertial Weil--Deligne type~$\tau_E$; if $E$ has additive,  potentially good reduction, let $e_E$ be its semistability defect.  Then $\tau_E$ is classified up to equivalence according to Table~\textup{\ref{MainTable}}.
\end{maintheorem*}

\begin{table} 
\normalfont\small\renewcommand{\arraystretch}{1.25}
\begin{tabular}{l|l|l|l|l|l} 
Reduction type & $p$ & $e_E$ & $v_p(\NE)$ & $\tau_E$ & Description \\ \hline\hline

good  & - & - & $0$ & trivial & trivial \\
\hline
multiplicative  & - & - & $1$ & $\tauspnew{p}$ & special \\
\hline
\multirow{3}{*}{\parbox[c]{15ex}{additive, \\ potentially \\ multiplicative}} & $\geq 3$ & - & $2$ & $\tauspnew{p} \otimes \epschar{p}$ & \multirow{3}{*}{special} \\
\cline{2-5}
& \multirow{2}{*}{$2$} & \multirow{2}{*}{-} & $4$ & $\tauspnew{2} \otimes \epschar{-4}$ \\
\cline{4-5}
& & & $6$ & $\tauspnew{2} \otimes \epschar{\pm 8}$ \\
\hline
\multirow{27}{*}{\parbox[c]{15ex}{additive, \\ potentially \\ good}} & \multirow{3}{*}{$\geq 5$} & $2$ & \multirow{3}{*}{$2$} & $\epschar{p}$ & \multirow{2}{*}{principal series} \\
\cline{3-3} \cline{5-5}
& & $3,4,6 \mid (p-1)$ & & $\tauPSnew{p}{1,1,e}$ \\
\cline{3-3} \cline{5-6}
& & $3,4,6 \mid (p+1)$ & & $\tauSCnew{p}{u,2,e}$ & \parbox[c][3ex]{15ex}{supercuspidal} \\
\cline{2-6}
 & \multirow{9}{*}{$3$} & $2$ & $2$ & $\epschar{3}$ & \multirow{2}{*}{principal series} \\
  \cline{3-5}
 & & $3$ & $4$ & $\tauPSnew{3}{1,2,3}$ & \\
 \cline{3-6}
 & & $3$ & $4$ & $\tauSCnew{3}{-1,2,3}$ & \multirow{2}{*}{supercuspidal} \\
 \cline{3-5}
 &  & $4$ & $2$ & $\tauSCnew{3}{-1,1,4}$ &  \\
 \cline{3-6}
 & & $6$ & $4$ & $\tauPSnew{3}{1,2,3} \otimes \epschar{3}$ & principal series \\
 \cline{3-6}
 & & $6$ & $4$ & $\tauSCnew{3}{-1,2,3} \otimes \epschar{3}$ &  \multirow{3}{*}{\parbox[c][6ex]{15ex}{supercuspidal}}\\
 \cline{3-5}
 & & \multirow{2}{*}{$12$} & $3$ & $\tauSCnew{3}{\pm 3,2,6}$ & \\
 \cline{4-5}
 & &  & $5$ & $\tauSCnew{3}{-3,4,6}_j$\ \ ($j=0,1,2$) \\
 \cline{2-6}
 & \multirow{17}{*}{$2$} & \multirow{2}{*}{$2$} & $4$ & $\epschar{-4}$ & 
 \multirow{2}{*}{principal series} \\
 \cline{4-5}
 & & & $6$ & $\epschar{\pm 8}$ &  \\
 \cline{3-6}
 &  & $3$ & $2$ & $\tauSCnew{2}{5,1,3}$ & supercuspidal \\
 \cline{3-6}
 & & \multirow{2}{*}{$4$} & \multirow{2}{*}{$8$} & \parbox[c][3ex]{30ex}{$\tauPSnew{2}{1,4,4} \otimes \epschar{d}$\ \  ($d=1,-4$)} & 
 \multirow{1}{*}{principal series}  \\
 \cline{5-6}
 & &  &  & \parbox[c][3ex]{30ex}{$\tauSCnew{2}{5,4,4} \otimes \epschar{d}$\ \  ($d=1,-4$)} & \multirow{8}{*}{\parbox[c][6ex]{15ex}{supercuspidal}} \\
 \cline{3-5}
 & & \multirow{2}{*}{$6$} & $4$ & $\tauSCnew{2}{5,1,3} \otimes \epschar{-4}$ & \\
 \cline{4-5}
 & &  & $6$ & \parbox[c][3ex]{20ex}{$\tauSCnew{2}{5,1,3} \otimes \epschar{\pm 8}$} &  \\
 \cline{3-5}
 & & \multirow{4}{*}{$8$} & $5$ &  \parbox[c][6ex]{20ex}{$\tauSCnew{2}{-4,3,4}$, \\ $\tauSCnew{2}{-20,3,4}$} \\
 \cline{4-5}
 & & & $6$ & \parbox[c][6ex]{22ex}{$\tauSCnew{2}{-4,3,4} \otimes \epschar{8}$, \\ $\tauSCnew{2}{-20,3,4} \otimes \epschar{8}$} &  \\
 \cline{4-5}
 & &  & 8 & \parbox[c][3ex]{30ex}{$\tauSCnew{2}{-4,6,4} \otimes \epschar{d}$\ \  ($d=1,-4$)} &  \\\cline{3-6}
 & & \multirow{5}{*}{$24$} & $3$  & $\tauex{1}$ & \multirow{4}{*}{\parbox[c]{15ex}{exceptional \\ supercuspidal}} \\
 \cline{4-5}
 & & & $4$ & $\tauex{1} \otimes \epschar{-4}$ \\
 \cline{4-5}
 & & & $6$ & $\tauex{1} \otimes \epschar{\pm 8}$ \\
 \cline{4-5}
 & & & $7$ & \parbox[c][3ex]{26ex}{$\tauex{2} \otimes \epschar{d}$ ($d=1,-4,\pm 8$)} 
\end{tabular}
\caption{\rule{0pt}{2.5ex}Inertial WD-types for elliptic curves over $\Q_p$}
\label{MainTable}
\end{table}

The notation in Table \ref{MainTable} is explained in section \ref{sec:notation}.  In Table~\ref{MainTable}, the exceptional (or primitive) supercuspidal representations are labelled as such and collected in the last rows of the table, whereas the nonexceptional (imprimitive) supercuspidal representations are labelled simply \emph{supercuspidal}, for brevity.

All types in Table~\ref{MainTable} arise for an elliptic curve over $\Q_p$: see Tables~\ref{table:types-Q3}, \ref{table:types-Q2:sc}, and~\ref{table:types-Q2:ex}.  See Dokchitser--Dokchitser \cite{dd-kodaira} for computation of Kodaira types---the type restricts the possible Kodaira types but not necessarily uniquely, so for simplicity we do not include them in our table.  

Our method of proof of the Main Theorem is by direct calculation: we deduce the inertial type associated to an elliptic curve over $\Q_p$ in terms of its reduction type.  We have endeavored to streamline these calculations while still remaining comprehensive and as self-contained as possible.  Many of our calculations are performed in the computer algebra system \textsf{Magma} \cite{Magma}; the code is available online \cite{ourcode}.  

Indeed, many of these calculations can be found in other places in the literature: for example, the $3$-adic types are already implicitly given in the proof of the modularity theorem \cite{bcdt01}, and Dieulefait--Pacetti--Tsaknias~\cite{DPT} more generally identify local invariants of Galois orbits of classical newforms (relevant here for weight $k=2$).  Coppola \cite{Cop20a,Cop20b} has recently studied wild Galois representations ($p=3$ and $e=12$; $p=2$ and $e=8,24$) over more general local fields, classifying the Galois representation up to isomorphism (as an abstract group) but without providing the explicit description of the underlying field.  Finally, Barrios--Roy \cite{BR22} recently studied representations (including the inertial type) for elliptic curves with nontrivial odd torsion.

This proof has an interesting algorithmic consequence: to compute the type of an elliptic curve $E$ over $\Q_\ell$, it suffices to run Tate's algorithm on $E$ over an explicit finite list of extensions of $\Q_\ell$, see e.g.\ Corollary~\ref{cor:algQ3}.  For reliability, we have implemented this algorithm on over thousands of elliptic curves over $\Q$; the code is available online \cite{ourcode}.  Dokchitser--Dokchitser also implemented algorithms for working with the Galois representations attached to elliptic curves over local fields in \textsf{Magma} \cite{Magma} by reconstructing representations from their (good) Euler factors \cite{dd-crelle} and a tame local reciprocity formula of Newton \cite{Newton12}.  This implementation does not compute the inertia field (necessary to distinguish the types) and runs into difficulties exactly in the complicated inertial cases we consider here (e.g.\ when there are multiple faithful representations of the inertia group in dimension $2$).  

\subsection{Outline}

The paper is organized as follows.  In sections \ref{S:recTypes}--\ref{S:recCurves}, we establish background by briefly reviewing some facts concerning $2$-dimensional Weil--Deligne representations, inertial types, and elliptic curves.  In section \ref{sec:potmotred}, we compute types in the case of potentially multiplicative reduction for all primes $p$ and for additive, potentially good reduction for $p \geq 5$.  The remainder of the paper is concerned with additive, potentially good reduction first for $p=3$ (section \ref{sec:inert3}) then $p=2$ (sections \ref{sec:inert2nonexcept}--\ref{S:exceptionalThm} for the nonexceptional and exceptional cases).

\subsection{Acknowledgements}

The authors would like to thank Fred Diamond and Panagiotis Tsanknias for many instructive conversations.  Freitas was supported by the
European Union's Horizon 2020 research and innovation programme under the Marie Sk\l{l}odowska-Curie grant 
agreement No.\ 747808 and partly supported by the grant {\it Proyecto RSME-FBBVA $2015$ Jos\'e Luis Rubio de Francia.}
Voight was supported by an NSF CAREER Award (DMS-1151047) and a Simons Collaboration Grant (550029) and would like to thank the Henri Lebesgue Center for its hospitality during the conference \emph{$p$-adic Langlands Correspondence: a Constructive and Algorithmic Approach} in September 2019.

\section{Two-dimensional Weil--Deligne representations}
\label{S:recTypes}

In this section, we quickly recall background on Galois representations of local fields and types. 
Our main references are Tate \cite{Tate79}, Rohrlich \cite{roh94}, Carayol \cite{car86}, and Bushnell--Henniart \cite{bh06}.

\subsection{Notation}

A \defi{(complex) quasicharacter} of a topological group $G$ is a continuous homomorphism $\chi \colon G \to \C^\times$ with open kernel; if further $\left|\chi(g)\right|=1$ for all $g \in G$, we call $\chi$ a \defi{(unitary) character}.   If $\chi \colon G \to \C^\times$ is a (quasi)character and $\varphi \colon G \to G'$ is a continuous group homomorphism, we say that $\chi$ \defi{factors through} $\varphi$ if there exists a (quasi)character $\chi' \colon G' \to \C^\times$ such that $\chi = \chi' \circ \varphi$.  

\begin{remark} \label{rmk:Cp}
Throughout, one can equally well replace $\C$ with any algebraically closed field of characteristic $0$.  There is also a well-behaved theory of types over $\Q_\ell$ with $\ell \neq p$---and even one over $\Q_p$, with additional effort.  
\end{remark}

Let $p$ be prime and let $F \supseteq \Q_p$ be a finite extension with algebraic closure~$\Fbar$ and maximal unramified extension $F\spun \subset \Fbar$.  Let $\calO_F \subset F$ be the valuation ring of $F$ with maximal ideal $\fp$, uniformizer $\pi \in \fp$, and residue field $k$ of cardinality $q \colonequals \#k$.  Let $\vv \colon F^\times \to \Z$ denote the valuation of $F$ normalized with $\vv(\pi)=1$, and let $|{\cdot}|_v \colon F^\times \to \R_{>0}^\times$ be the associated normalized absolute value. 
Let $W_F < \Gal(\Fbar\,|\,F)$ be the Weil group of $F$ and $I_F < W_F$ its inertia subgroup, fitting into the exact sequence
\begin{equation} 
1 \to I_F \to W_F \to \Z \to 1. 
\end{equation}
 For $F = \Q_p$, for brevity we replace $F$ by $p$ in the subscript, e.g., writing $I_p < W_p$.  

Let $W_F^{\textup{ab}}$ denote the maximal abelian quotient of $W_F$ and let
$\Art_F \colon F^\times \xrightarrow{\sim} W_F^{\textup{ab}}$ be the \defi{Artin reciprocity map} from local class field theory, the isomorphism of topological groups 
sending $\pi \in \calO_F$ to the class of a  
\emph{geometric} Frobenius element $\Fr \in W_F^{\textup{ab}}$ characterized by $\Fr(x^q)=x$ for $x \in k$.  The map $\Art_F$ allows us to identify 
a (quasi)character $\chi$ of $W_F$ with the (quasi)character $\chi\spArt \colonequals \chi \circ \Art_F$ 
of $F^\times$, and conversely.  
The \defi{conductor} of $\chi$ is the ideal $\cond(\chi) \colonequals \fp^m$ where 
$\condexp(\chi) \colonequals m \in \Z_{\geq 0}$ is \defi{conductor exponent}, the smallest nonnegative integer such that the restriction $\chi\spArt|_{1+\fp^m}$ to $1+\fp^m \leq \calO_K^\times$ is trivial.  Let $\omega \colon W_F \to \C^\times$ be the quasicharacter corresponding to the norm quasicharacter~$|{\cdot}|_v$, so that~$\omega(g)=q^{-a}$ for $g|_{F\spun}=\Fr^{a}$ with $a \in \Z$.

\begin{definition} \label{defn:WDRep}
A ($n$-dimensional) \defi{Weil--Deligne representation} is a pair $(\rho,N)$ such that:
\begin{enumroman}
 \item $\rho \colon W_F \to \GL_n(\C)$ is a homomorphism with open kernel; and
 \item $N \in \GL_n(\C)$ is nilpotent and satisfies
 \begin{equation}
  \rho(g)N\rho(g)^{-1} = \omega(g)N \quad \text{ for all } g \in W_F.
 \end{equation}
\end{enumroman}
An \defi{isomorphism} (or \defi{equivalence}) of Weil--Deligne representations from $(\rho,N)$ to $(\rho',N')$ is specified by an element $P \in \GL_n(\C)$ such that $\rho'(g)=P\rho(g)P^{-1}$ for all $g \in W_F$ and $N'=PNP^{-1}$.
\end{definition}

\begin{remark}
The nilpotent element $N$ comes from the fact that there exists an open subgroup $H \leq I_F$ such that $\rho|_H$ is unipotent; below, we will usually have $N=0$.
\end{remark}

\subsection{Classification}

Every $2$-dimensional Weil--Deligne representation arises up to isomorphism from one of the following three possibilities.  

\begin{itemize}
\item \emph{Principal series}.  Let $\chi_1, \chi_2 \colon W_F \to \C^\times$ be quasicharacters such that $\chi_1 \chi_2^{-1} \ne~\omega^{\pm 1}$.
The \defi{principal series} representation associated to~$\chi_1, \chi_2$ is~$(\PS(\chi_1, \chi_2), 0)$, where 
\[ \PS(\chi_1,\chi_2) \colonequals \chi_1 \oplus \chi_2. \]
Its conductor exponent is given by 
\begin{equation} \label{E:condPS}
 \condexp(\PS(\chi_1, \chi_2)) = \condexp(\chi_1) + \condexp(\chi_2).
\end{equation}

\item \emph{Special or Steinberg representations}.  Let $\chi \colon W_F \to \C^\times$ be a quasicharacter.   
The \defi{special} or \defi{Steinberg} representation associated to $\chi$ is~$(\Sp(\chi), N)$, where 
\begin{equation} \label{eqn:Spchi}
 \Sp(\chi) \colonequals \chi \omega \oplus \chi \quad \text{ and } \quad
 N = \begin{pmatrix} 0 & 1 \\ 0 & 0 \end{pmatrix}.
\end{equation}
We have
\begin{equation}\label{E:condSS}
\condexp(\Sp(\chi)) =
\begin{cases}
2\condexp(\chi), & \text{ if $\chi$ is ramified;} \\
1, & \text{ otherwise.}
\end{cases} 
\end{equation}

\item \emph{Supercuspidal representations}. The  Weil--Deligne representations $(\rho, 0)$ where $\rho$ is an \emph{irreducible} $2$-dimensional representation of $W_F$ are called \defi{supercuspidal}.  Supercuspidal representations are classified by their projective images in $\PGL_2(\C)$ (see Bushnell--Henniart \cite[sections 41 and 42]{bh06} or Carayol \cite[section 12]{car86}).  We say that
$\rho$ is \defi{nonexceptional} (or \defi{imprimitive}) if its projective image is dihedral, otherwise $\rho$ is \defi{exceptional} (or \defi{primitive}) and has projective image $A_4$ or $S_4$.  (Since $W_F$ is totally disconnected, the projective image $A_5$ cannot occur.)  
\end{itemize}

\subsection{Nonexceptional supercuspidal representations}

Since they will command significant attention here, we explore supercuspidal representations further.  We begin with the nonexceptional supercuspidal representations.

Let $\Mquad \supset F$ be a quadratic extension and let $\psi_K \colon W_F \to \{\pm 1\}$ be the quadratic character of $W_F$ with kernel $W_K$.  Let $\chi \colon W_{\Mquad} \rightarrow~\C^\times$ 
be a quasicharacter and consider the associated quasicharacter $\chi\spArt \colon K^\times \to \C^\times$.  
Let $s \in W_F$ be a lift of the nontrivial element in $\Gal(\Mquad\,|\,F)$.  Since $W_K \trianglelefteq W_F$ is normal (with $s$ representing the nontrivial coset), the \defi{$s$-conjugate} of $\chi$ 
\begin{equation} \label{eqn:chisconj}
\begin{aligned}
\chi^s \colon W_{\Mquad} &\rightarrow \C^\times \\ 
\chi^s(g) &= \chi(s^{-1}gs), 
\end{aligned}
\end{equation}
is independent of the choice of $s$.  By local class field theory, we have $(\chi^s)\spart=\chi\spart \circ s$.

\begin{lemma} \label{L:factorvianorm2}
The following are equivalent:
\begin{enumroman}
\item $\chi^s=\chi$;
\item $s(x)/x \in \ker \chi\spArt$ for all $x \in K^\times$; and
\item $\chi\spArt$ factors through the norm map $\Nm_{K|F} \colon K^\times \to F^\times$.
\end{enumroman}
\end{lemma}

Of course, since $K^\times$ is generated by $\pi$ (for any choice of uniformizer $\pi$) and $\calO_K^\times$, it is enough to check (ii) for $x=\pi$ and all $x \in \calO_K^\times$.  

\begin{proof} 
For (i) $\Leftrightarrow$ (ii), we observe
\begin{equation} \label{eqn:chischi}
\begin{aligned}
\chi^s = \chi &\quad \Leftrightarrow \quad (\chi^s)\spArt(x) = \chi\spArt(x) \quad \text{ for all $x \in K^\times$}\\
&\quad \Leftrightarrow \quad \chi\spArt(s(x)/x) = 1 \quad \text{ for all $x \in K^\times$}\\
&\quad \Leftrightarrow \quad \{s(x)/x : x \in K^\times\} \leq \ker \chi\spArt.\\
\end{aligned}
\end{equation}
For (ii) $\Leftrightarrow$ (iii), we recall that $\ker(\Nm_{K|F}) = \{s(x)/x : x \in K^\times\}$ by Hilbert's Theorem 90, so (ii) holds if and only if $\chi\spArt$ factors through the surjective norm map $\Nm_{K|F} \colon K^\times \to \Nm_{K|F}(K^\times)$.  But of course $\Nm_{K|F}(K^\times) \leq F^\times$ is open of finite index, so $\theta'$
extends to a quasicharacter $\theta \colon F^\times \to \C^\times$, and we see that (ii) $\Leftrightarrow$ (iii).
\end{proof}

We record the following consequences.

\begin{corollary}\label{cor:factorvianorm}
The following statements hold.
\begin{enumalph}
\item If $\chi\spArt(-1) = -1$ then~$\chi\spArt$ does not factor via the norm map.
\item Suppose $\calO_F^\times \leq \ker \chi\spArt$.  Then~$\chi\spArt$ factors through the norm map if and only if $\chi\spArt(s(\pi)/\pi) = 1$ and $\chi\spArt|_{\calO_K^\times}$ is quadratic.
\item The character $\left(\chi^s /\chi\right)\spArt$ factors through the norm map if and only if $\chi^s /\chi$ is quadratic. 
\end{enumalph}
\end{corollary}

\begin{proof} 
For (a), if $\chi\spArt  = \theta \circ \Nm_{K|F}$ then $\chi\spArt(-1)=\theta((-1)^2)=1$.
For (b), we apply the equivalence (ii) $\Leftrightarrow$ (iii) of Lemma \ref{L:factorvianorm2}.  For $u \in \calO_K^\times$, by hypothesis we have
\[ \chi\spArt(s(u)/u) = \chi\spArt(\Nm_{K|F}(u)/u^2)= 1/\chi\spArt(u)^2 \]
so $s(u)/u \in \ker \chi\spArt$ if and only if $\chi\spArt(u)^2=1$.  Thus $s(x)/x \in \ker \chi\spArt$ for all $x \in K^\times$ if and only if $\chi\spArt(s(\pi)/\pi)=1$ and $(\chi\spArt|_{\calO_K^\times})^2$ is trivial, i.e., $\chi\spArt|_{\calO_K^\times}$ is quadratic.  

For (c), we apply the equivalence (i) $\Leftrightarrow$ (iii): so $\left(\chi^s /\chi\right)\spArt$ factors through the norm map if and only if $(\chi^s/\chi)^s = \chi^s/\chi$.  But $(\chi^s/\chi)^s = \chi/\chi^s = (\chi^s/\chi)^{-1}$, so the condition holds if and only if $\chi^s/\chi=(\chi^s/\chi)^{-1}$ if and only if $\chi^s/\chi$ is quadratic.
\end{proof}

Suppose that $\chi \neq \chi^s$. Then the \defi{nonexceptional supercuspidal} representation attached to $\chi$ is~$(\Ind_{W_{\Mquad}}^{W_F}\chi, 0)$, the induction of $\chi$ from $W_{\Mquad}$ to $W_F$ (with $N=0$).  
(The condition $\chi \neq \chi^s$ is necessary to ensure that $\BCKF \chi$ is irreducible.)
Recalling that $\psi_K$ is the quadratic character of $W_F$ corresponding to $K$, we have
\begin{equation} \label{E:condBC}
\condexp(\BCKF \chi) = \begin{cases}
		      2\condexp(\chi),  & \text{ if ${\Mquad}\,|\,F$ is unramified;}\\
                      \condexp(\chi) + \condexp(\psi_K),  & \text{ if $\Mquad\,|\,F$ is ramified. }\                                                         
                  \end{cases}
\end{equation}

We conclude with a few simple lemmas.  

\begin{lemma} \label{lem:order4max}
If $\chi^s|_{I_K} = \chi^{-1}|_{I_K}$, then $(\chi^s/\chi)|_{I_K}$ has order dividing $2$ if and only if~$\chi|_{I_K}$ has order dividing~$4$.
\end{lemma}

\begin{proof} 
Indeed, we have $(\chi^s/\chi)|_{I_K} = \chi^{-2}|_{I_K}$.
\end{proof}
 
\begin{lemma} \label{lem:centralChar}
Let $\delta \colonequals \det(\Ind_{W_{\Mquad}}^{W_F}\chi)$.  Then 
\[ \delta\spArt = \psi_K\spArt \chi\spArt|_{F^\times}. \]
In particular, if $\delta=\omega$ then $\psi_K\spArt \chi\spArt|_{F^\times} = |\,|_v$ and 
\[ \chi\spArt|_{\calO_F^\times} = \psi_K\spArt|_{\calO_F^\times} \quad \text{and} \quad \chi^s|_{I_K} = \chi^{-1}|_{I_K}. \]
\end{lemma}

\begin{proof}
We have \cite[\S 29.2]{bh06} (with $d=2$, $m=1$)
\begin{equation} \label{eqn:detpsik}
\delta \colonequals \det(\Ind_{W_{\Mquad}}^{W_F}\chi) = \psi_K \cdot (\chi \circ \ver_{W_F|W_K}) 
\end{equation}
where $\ver_{W_F|W_K}$ is the transfer (or Verlagerung) map, fitting in the fundamental commutative diagram from local class field theory:
\begin{equation}
\begin{aligned}
\xymatrix{
F^\times \ar[r] \ar@{^{(}->}[d] & W_F^{\textup{ab}} \ar[d]^{\ver_{W_F|W_K}} \\
K^\times \ar[r] & W_K^{\textup{ab}}
} 
\end{aligned}
\end{equation}
(The equality \eqref{eqn:detpsik} can also be verified in this case by a direct calculation.)  Thus 
\begin{equation}
\delta\spArt = \psi_K\spArt \chi\spArt|_{F^\times}.
\end{equation}

In particular, we have $\delta = \omega$ if and only if $\delta\spArt = \psi_K\spArt \chi\spArt|_{F^\times} = |\,|_v$.  When $\delta=\omega$, since~$\psi_K$ is quadratic we have $\chi\spArt|_{\calO_F^\times}=(\psi_K^{-1})\spArt|_{\calO_F^\times} = \psi_K\spArt|_{\calO_F^\times}$; and as in Lemma \ref{L:factorvianorm2} we have $(\chi^s \chi)\spArt = \chi\spart \circ \Nm_{K|F}$ so restricting to $\calO_K^\times$ gives $\chi^s|_{I_K}=\chi^{-1}|_{I_K}$.
\end{proof}

\subsection*{Exceptional supercuspidal representations} \label{S:exceptional} 
The remaining supercuspidal representations are exceptional (primitive).  A supercuspidal Weil-Deligne representation $\rho\colon W_F \to \GL_2(\C)$ is called \defi{exceptional} or \defi{primitive} if the image
of the projective representation $\mathrm{P}\rho$ is isomorphic to $A_4$ or $S_4$. Exceptional representations only exist in residual characteristic $p=2$, so suppose that $p=2$.  
We will also use the following characterization.  For $\rho$ supercuspidal, let $\mathfrak{I}(\rho)$ be the group of cha\-rac\-ters $\xi \colon W_F \to \C^\times$ such that $\rho \otimes \xi \simeq \rho$.  
 
\begin{proposition}\label{prop:primitive} 
Let $\rho$ be a supercuspidal Weil--Deligne representation.  Then, $\rho$ is exceptional if and only if $\#\mathfrak{I}(\rho) \in \{1,2,4\}$.
\end{proposition}

\begin{proof}
See Bushnell--Henniart \cite[\S 41.3]{bh06}.
\end{proof}

This characterization also refines the imprimitive case, as follows.  We write $D_n$ for the dihedral group of order $2n$ and $C_n$ for the cyclic group of order $n$.

\begin{proposition}\label{prop:triply-imprimitive} 
Let $\rho = \Ind_{W_K}^{W_F}\chi$ be an imprimitive representation, 
where $K \supset F$ is a quadratic extension and $\chi \colon W_K \to \C^\times$ a character such that $\chi^s \neq \chi$.  Then the following statements hold:
\begin{enumalph}
\item $\#\mathfrak{I}(\rho) = 2$ if and only if $\rho$ is induced from a \emph{unique} quadratic extension $K \supset F$.
\item $\#\mathfrak{I}(\rho) = 4$ if and only if $\rho$ has projective image isomorphic to $D_2 \simeq C_2 \times C_2$ and can be induced from three distinct quadratic extensions if and only if $\chi^s/\chi$ factors through the norm map.
\end{enumalph}
\end{proposition}

\begin{proof} 
See Bushnell--Henniart \cite[\S 41.3, Corollary]{bh06} and G\'erardin \cite[Section~2.7]{ger78}.
\end{proof}

Following Propositions \ref{prop:primitive} and \ref{prop:triply-imprimitive}, we say that a supercuspidal representation $\rho$ is \defi{primitive}, \defi{simply imprimitive}, or \defi{triply imprimitive} according as $\#\mathfrak{I}(\rho)=1,2,4$.

\begin{proposition}\label{prop:triply-imprimitive:2} Let $\rho$ be a primitive representation.  
\begin{enumalph}
\item There exists a cubic extension $L\,|\,F$ such that $\rho|_{W_L}$ is imprimitive.
\item If $L\,|\,F$ is cubic Galois, then the representation $\rho|_{W_L}$ is triply imprimitive.
\item If $L\,|\,F$ is cubic non Galois, let $M\,|\,F$ be the normal closure of $L\,|\,F$ and $E\,|\,F$ the maximal unramified sub-extension of $M\,|\,F$. 
Then, the representation $\rho|_{W_L}$ is simply imprimitive, $\rho|_{W_M}$ is triply imprimitive, and $\rho|_{W_E}$ is primitive.
\end{enumalph}
\end{proposition}

\begin{proof} See Bushnell-Henniart~\cite[\S 42.2, Theorem, p.\ 258]{bh06}.
\end{proof}

Let $L \supseteq F$ be a (tamely) ramified cubic extension, and let $M \supset L$ be a ramified quadratic extension.  (The condition that $M$ is ramified over $L$ is indeed necessary \cite[\S 42.1, Proposition, p.\ 257, part (1)]{bh06}.)  Let  $\chi$ be a character of $W_M$ such that $\chi\spArt$ does not factor through the norm map $\Nm_{M|L}$.
Given the data $(L,M,\chi)$, by Bushnell--Henniart \cite[p.~261]{bh06} there is a \defi{exceptional supercuspidal} Weil--Deligne representation $(\rho,0)$ such that 
\begin{equation}
 \rho|_{W_L} = \Ind_{W_M}^{W_L}\chi.
 \end{equation}
 Conversely, every exceptional supercuspidal representation is uniquely determined by such a triple $(L,M,\chi)$, up to equivalence \cite[Lemme 12.1.3]{car86}.

\subsection{Inertial types}\label{S:inertialTypes} 
We now study the restriction to inertia.  An \defi{inertial Weil--Deligne (or $\WD$-)type} is an equivalence class $[\rho,N]$ of Weil--Deligne representations~$(\rho, N)$ under the equivalence relation $(\rho, N) \sim (\rho', N')$ if and only if there exists $P \in \GL_2(\C)$ such that $\rho'(g)=P\rho(g)P^{-1}$ and $N'=PNP^{-1}$ for all $g \in I_F$.  The content is in the restriction to $g \in I_F$; we might think of this as being an equivalence of Weil--Deligne representations over $F\spun$.  Such an equivalence class is determined by the pair $(\tau,N)$ where  $\tau = \rho|_{I_F}$ is the (common) restriction to $I_F$ for a WD-type, with the evident notion of equivalence,
so this definition agrees with the one given in the introduction.  Except for the special (Steinberg) representations we have $N = 0$, so (aside from section~\ref{S:multiplicative}) we drop $N$ from the notation and write simply $\tau$.

We record the following classification of all inertial WD-types. 

\begin{proposition}\label{prop:non-prim-inert-type} 
Let $\tau\colon I_F \to \GL_2(\C)$ be an inertial $\WD$-type. Then exactly one of the following holds:
\begin{enumroman}
\item $\tau$ is the restriction of a principal series, i.e., there exist $\chi_1, \chi_2\colon W_F \to \C^\times$ such that
$$\tau \simeq \PS(\chi_1, \chi_2)|_{I_F} =
\chi_1|_{I_F} \oplus \chi_2|_{I_F};$$
\item $\tau \simeq \Sp(\chi)|_{I_F}$ is the restriction of a special series for $\chi$ a character of $W_F$; 
\item There exists a character $\chi\colon W_{\Mquad} \to \C^\times$, where ${\Mquad} \supset F$ is the unramified quadratic extension, such that $\chi \ne \chi^s$ and 
$$\tau \simeq (\Ind_{W_{\Mquad}}^{W_{F}}\chi)|_{I_F} = \chi|_{I_F} \oplus \chi^s|_{I_F};$$
\item There exist a ramified quadratic extension ${\Mquad} \supseteq F$ and a character $\chi\colon W_{\Mquad} \to  \C^\times$ such that $\chi|_{I_K} \ne \chi^s|_{I_K}$ and
$$\tau \simeq  \Ind_{I_{\Mquad}}^{I_{F}}(\chi|_{I_K})$$
is irreducible; or
\item $\tau$ is the restriction of an exceptional supercuspidal Weil--Deligne representation.
\end{enumroman}
\end{proposition}
\begin{proof} For (i)--(iv),
see Breuil--M\'ezard \cite[Lemme 2.1.1.2, Th\'eor\`eme~2.1.1.4]{bm02} and for (v) see 
Bushnell--Henniart \cite[\S 41 and \S 42]{bh06}.
\end{proof}

By Proposition \ref{prop:non-prim-inert-type}, the classification of Weil--Deligne representations in the previous section remains well-defined on inertial types, and so we may accordingly say $[\rho, N]$ is \defi{principal series}, \defi{special}, or \defi{(nonexceptional or exceptional) supercuspidal}. 

\subsection{Notation} \label{sec:notation}

We conclude this section with the notation we will use throughout, in one place for convenience.
 
\begin{itemize}
\item We write $\epschar{d} \colon I_p \to \C^\times$ for the character associated to the (ramified) quadratic extension $\Q_p(\sqrt{d})$ of discriminant $d \in \Z_p$ (more precisely, $d \in \Z_p/\Z_p^{\times 2}$).  
\item We write $\tauspnew{p}$ to denote the special (Steinberg) type \eqref{eqn:Spchi} (not including the nilpotent monodromy operator $N$ in the notation); in all other cases, $N=0$.
\item To identify the nonexceptional supercuspidal types, we use the notation 
\begin{equation} \label{eqn:scnot}
\tau_{\textup{sc},p}(d,f,r)_j \colonequals \bigl(\Ind_{W_{\Q_{p}(\sqrt{d})}}^{W_{\Q_p}} \chi_{(d,f,r)} \bigr)|_{I_p} 
\end{equation}
and for the principal series types we use
\begin{equation} \label{eqn:psnot}
\tau_{\textup{ps},p}(1,f,r)  \colonequals \chi_{(1,f,r)}|_{I_p} \oplus \chi_{(1,f,r)}^{-1}|_{I_p}
\end{equation}
where:
\begin{itemize}
\item $d$ is the discriminant of $K \colonequals \Q_{p}(\sqrt{d})$, with $[K:\Q_p] \leq 2$; 
\item for $p \neq 2$, let $u \in \Z_p^\times \smallsetminus \Z_{p}^{\times 2}$ be a nonsquare, so $\Q_p(\sqrt{u})$ is the unramified quadratic extension of $\Q_p$ (having discriminant $u$);
\item $\chi_{(d,f,r)} \colon W_K \to \C^\times$ is a character, where:
\begin{itemize}
\item $f$ is the conductor exponent of the character $\chi$ (as a power of the maximal ideal in the ring of integers of $K$);
\item $r$ is the order of the character~$\chi$ on the inertia subgroup $I_K\subset W_K$; and 
\item $j$ is an additional label (only needed for $p=3$, see Table \ref{table:Q3}).
\end{itemize}
\end{itemize}
\item For $p=2$, two exceptional (octahedral) representations $\tauex{i}$ for $i=1,2$ are explicitly given (see section \ref{sec:explchar2}).
\end{itemize}

\section{Background on elliptic curves}
\label{S:recCurves}

In this section we organize some facts about elliptic curves and provide a few preliminary results on their inertial types.
Throughout this section, let $E$ be an elliptic curve 
over the $p$-adic field $F$, and let $N_E$ be the conductor of $E$.

\subsection{Inertial types} \label{sec:3point1}

We begin by defining a Weil--Deligne representation
$(\rho_{E},N)$ attached to~$E$: for complete details, we refer to Rohrlich \cite[\S 4 and \S 13--15]{roh94}.   
We start with the representation
$\rho_{E,\ell}  \colon \Gal(\Fbar\,|\,F) \to \GL_2(\Q_\ell)$ defined by the action of $\Gal(\Fbar\,|\,F)$ on the \'etale cohomology group $H_{\textup{et}}^1(E \times_F \Fbar,\Q_\ell) \simeq \Z_\ell^2$ for some prime $\ell \neq p$.  We may also work \emph{dually} with the $\ell$-adic Tate module, via the isomorphism
\begin{equation}
T_\ell(A) \colonequals \varprojlim_n E[\ell^n] \simeq H_{\textup{et}}^1(E \times_F \Fbar,\Q_\ell)\spcheck. 
\end{equation}
The determinant of this representation is the cyclotomic character; so in the principal series case we have $\chi_2=\chi_1^{-1}\cdot |\,|_F$ and in the supercuspidal case, Lemma \ref{lem:centralChar} applies.

Next, we consider two cases.
\begin{itemize}
\item If $E$ has potentially good reduction, then $\rho_{E,\ell}(I_F)$ has finite order.  We take $N=0$ and $\rho_E$ is obtained by extension of scalars of the restriction~$\rho_{E,\ell}|_{W_F}$ via an embedding $\iota\colon \Q_\ell \hookrightarrow \C$; the $\C$-equivalence class is well-defined, independent of choices.  (See also Remark~\ref{rmk:Cp}.)
\item Otherwise, $E$ has potentially multiplicative reduction, and so $\rho_{E,\ell}(I_F)$ is infinite.  Then $E$ obtains split multiplicative reduction over an at most quadratic extension $K$.  Let $\chi$ be the at most quadratic character of $W_F$ attached to $K$.  Then we take
$N = \begin{pmatrix} 0 & 1 \\ 0 & 0 \end{pmatrix}$ and $\rho_E = \Sp(\chi)$ the Steinberg representation attached to $\chi$.
\end{itemize}

In either case, we define the \defi{inertial WD-type}~$\tau_E$ 
of $E$ to be the equivalence class $\tau_E= [\rho_E,N]$ as defined in section~\ref{S:inertialTypes}. Finally, we note that the conductors of~$\rho_{E,p}$ and~$\tau_E$ are both equal to $N_E$, the conductor of $E$ (see e.g.~Rohrlich \cite[\S 18]{roh94} or~Darmon--Diamond--Taylor \cite[Remark~2.14]{DDT}).  

\begin{example}
If $E$ already has good reduction over $F$, then $\tau_E$ is trivial.
\end{example}

\begin{example} \label{exm:twist}
For a ramified quadratic extension $\Q_p(\sqrt{d}) \supseteq \Q_p$ of discriminant~$d$, let $E_d$ be the quadratic twist of~$E$ over $\Q_p$ by~$d$.  Then $\tau_{E_d} \simeq \tau_E \otimes \epschar{d}$,  
and the nilpotent operator $N$ remains unchanged as the finiteness of the image of inertia is an invariant up to twist.
\end{example}

The following summarizes the well-known possibilites for~$N_E$ in the case $F=\Q_p$.

\begin{lemma} \label{L:condBound} 
Let $E$ be an elliptic curve over $\Q_p$.  Then
\[ 0 \leq \ord_p(\NE) \leq 
\begin{cases}
2, & \text{ if $p \geq 5$}; \\
5, & \text{ if $p=3$}; \\
8, & \text{ if $p=2$.}
\end{cases} \]
Moreover, if $E$ has additive reduction then $\ord_p(N_E) \geq 2$.
\end{lemma}

\begin{proof} See e.g.\ Silverman \cite[Theorem~IV.10.4]{SilvermanII}.
\end{proof}

\begin{remark}
Elliptic curves defined over \emph{ramified} extensions of $\Q_2$ or $\Q_3$ may have conductors  whose valuations are higher than those given by Lemma~\ref{L:condBound}. 
\end{remark}

\subsection{Potentially good reduction}

In this section, suppose that $E$ over $F$ has potentially good reduction, and let $\tau=\tau_E$ be its inertial type (with $N=0$).  In this section, we set up some preparatory facts needed in the sequel.  
 
Let $m \in \Z_{\geq 3}$ be coprime to~$p$, and let $L \colonequals \Fun(E[m])$. 
The extension~$L$ is independent of~$m$ 
(see Serre-Tate~\cite[\S 2, Corollary~3]{ST1968})
and it has two other equivalent descriptions:
\begin{itemize}
\item $L$ is the minimal extension of $\Fun$ where $E$ achieves good reduction; and
\item $L$ is the fixed field of $\ker \tau$.
\end{itemize}
We call $L$ the \defi{inertial field} of $E$.  Write $\Phi \colonequals \Gal(L\,|\,\Fun)$ and define the \defi{semistability defect} of $E$ to be $e = e_E \colonequals \#\Phi$.  
The field $L$ is the compositum of $\Fun$ by a minimal totally ramified extension of $F$ of degree $e$ where $E$ obtains good reduction (though this extension need not be Galois over $F$), so this definition agrees with the one in the introduction.  

The semistability defect can be computed as the ramification degree of $\Q_p(E[m])$, realized using division polynomials.  

The following describes the possibilities for~$\Phi$.

\begin{lemma} \label{L:phi} 
Exactly one of the following possibilities hold.
\begin{enumroman}
 \item $\Phi$ is cyclic of order $2,3,4,6$;
 \item $p=3$ and $\Phi \simeq \Z/3 \rtimes \Z/4$ is of order $12$;
 \item $p=2$ and $\Phi \simeq Q_8$ is isomorphic to a quaternion group of order $8$; or
 \item $p=2$ and $\Phi \simeq \SL_2(\F_3)$ is of order $24$.
\end{enumroman}
\end{lemma}
\begin{proof}
See Kraus \cite[pp.\ 354--357]{kra90}.
\end{proof}

\begin{corollary}
Let $p \geq 5$ and suppose $E$ has $p^2 \mid N_E$.  Then the semistability defect $e$ is the smallest integer $e \in \{2,3,4,6\}$ such that $E$ obtains good reduction over $\Q_p(\sqrt[e]{p})$.  
\end{corollary}

\begin{proof}
We are in case (i), so $\Phi$ is cyclic and the inertial field is tame so given by $\Q\spun_p(\sqrt[e]{p})$; but good reduction is invariant under unramified extensions, so it suffices to check over $\Q_p(\sqrt[e]{p})$.
\end{proof}

The above lemma crucially restricts the possibilities in the ramified case, as follows.

\begin{cor} \label{cor:notfactorbyinertia}
Suppose $\tau_E \simeq \Ind_{I_{\Mquad}}^{I_{F}}(\chi|_{I_K})$ where $\Mquad \supset F$ is a ramified quadratic extension and $\chi \colon W_{\Mquad} \to \C^\times$ is a character that does not factor through the norm.  Then either $p=2$ or (\/$p=3$ and $e=12$).
\end{cor}

\begin{proof}
If $p \geq 5$, then by Lemma~\ref{L:phi}, $\Phi$ is cyclic, and so $\tau_E$ is reducible.  But  by Proposition~\ref{prop:non-prim-inert-type}(iv), $\tau_E$ is irreducible, a contradiction.  If $p=3$, then we must have case (ii) in Lemma~\ref{L:phi}.
\end{proof}

Now suppose that $F=\Q_p$.  The next lemma already determines $\tau$ in the first nontrivial case.   

\begin{lemma} \label{L:e=2}
Suppose $E$ over $F=\Q_p$ has semistability defect $e_E=2$.  Then $\tau_E$ is principal series, and the following statements hold.  
\begin{enumalph}
 \item If $p \geq 3$, then $N_E=p^2$ and $\tau_E \simeq \epschar{p}$.  
 \item If $p =2$, then $N_E=2^4,2^6$, and 
 \[ \tau_E \simeq 
 \begin{cases}
 \epschar{-4}, & \text{if $N_E = 2^4$;} \\
 \epschar{\pm 8}, & \text{if $N_E = 2^6$.}
 \end{cases} \]
 \end{enumalph}
\end{lemma}

\begin{proof} 
The hypothesis that $e=e_E=2$ means that there exists a ramified quadratic extension $\Q_p(\sqrt{d}) \supseteq \Q_p$ such that the quadratic twist $E_d$ (as in Example \ref{exm:twist}) has good reduction (see e.g.\ Freitas--Kraus~\cite[Lemmas~3--4]{fk17}) and therefore the inertial type $\tau_{E_d}$ is trivial.  Thus 
\[ \tau \simeq \tau_{E_d} \otimes \epschar{d} \simeq \epschar{d} \simeq \PS(\chichar{d},\chichar{d})|_{I_p} \] 
is principal series.  We have $d=p$ if $p \geq 3$; if $p=2$, we have $d=-4,\pm 8$.
The claim on the conductor follows from 
$\ord_p(N_E) = \condexp(\tau) = 2\cond (\chichar{d})$ by \eqref{E:condPS}.
\end{proof}

At the other extreme, we conclude this short section with a preliminary step in classifying exceptional supercuspidal types arising from elliptic curves over $\Q_2$; these will be given explicitly in section~\ref{S:exceptionalThm}.

\begin{lemma} \label{L:exceptional=24}
Suppose $F=\Q_2$ and $E$ has potentially good reduction.  Then $\tau_E$ is exceptional supercuspidal if and only if $e=24$.  
\end{lemma}

\begin{proof} 
Suppose $\tau$ is exceptional.  We look at the group structure of the image of the projective representation obtained by postcomposing with $\GL_2(\C) \to \PGL_2(\C)$.  By Bushnell--Henniart \cite[section~42.3]{bh06}, we must have semistability defect $e \geq 12$, so in fact $e=24$ by Lemma~\ref{L:phi}.

Conversely, suppose $e=24$, and let 
$\rho_{E,3} \colon W_2 \to \GL_2(\Q_3)$ and
$\rhobar_3 \colon W_2 \to \GL_2(\F_3)$ respectively be the $3$-adic and mod $3$ Galois representations associated to $E$, 
restricted to the Weil--Deligne group $W_2$.  By Dokchitser--Dokchitser \cite[Lemma~1]{dd08}, there is an unramified twist of $\rho_{E,3}$ factoring through the Galois group of $K \colonequals \Q_2(E[3])$ (over $\Q_2$), so the images of the projective representations $\mathrm{P}\rho_{E,3} \colon W_2 \to \PGL_2(\Q_3)$ and $\mathrm{P}\rhobar_{3} \colon W_2 \to \PGL_2(\F_3)$ are isomorphic as abstract groups.  By hypothesis and Lemma~\ref{L:phi}, $\tau$ is irreducible with image isomorphic to $\Phi \simeq \SL_2(\F_3)$, so $\rhobar_3(W_2) = \GL_2(\F_3)$ and $\rhobar_3$ is surjective \cite[Table~1]{dd08}.  Thus $\mathrm{P}\rho_{E,3}$ has image isomorphic to $\PGL_2(\F_3) \simeq S_4$, so $\rho_{E,3}$ (extended to $\GL_2(\C)$) and hence~$\tau_E$ is exceptional  supercuspidal.
\end{proof}

\section{Inertial types: uniform cases} \label{sec:potmotred}

We now embark on an explicit and complete description of the inertial types arising from elliptic curves over $\Q_p$.  In this section, we treat two cases where the answer is close to uniform in $p$: potentially multiplicative reduction and $p \geq 5$.
Throughout, we use the notation collected in 
section~\ref{sec:notation}.

\subsection{Potentially multiplicative reduction and special types}
\label{S:multiplicative}

We begin with a general result on inertial types for elliptic curves with potentially multiplicative reduction. These are the only types with a nonzero nilpotent operator as in section \ref{sec:3point1}.  

\begin{proposition}  \label{prop:Emult}
Let $E$ be an elliptic curve over $\Q_p$ with potentially multiplicative reduction, conductor $\NE$, and inertial type $\tau_E$.  
Then the following statements hold.
\begin{enumalph}
\item If $E$ has multiplicative reduction (over $\Q_p$), then $\NE=p$ and $\tau_E \simeq \tauspnew{p}$ is special. 
\item Suppose $E$ has additive (but potentially multiplicative) reduction.  Then $p^2 \mid \NE$, and $\tau_E$ is special.  Moreover: 
\begin{enumromanii}
 \item  If $p \geq 3$, then $\NE=p^2$ and $\tau_E \simeq \tauspnew{p} \otimes \epschar{p}$.
 \item  If $p = 2$, then $\NE=p^4$ or $\NE=p^6$, and
 \[ \tau_E \simeq 
 \begin{cases} 
 \tauspnew{2} \otimes \epschar{-4}, & \text{ if $\NE=p^4$;} \\
 \tauspnew{2} \otimes \epschar{\pm 8} \text{ or } \tauspnew{2} \otimes \epschar{\pm 8}, & \text{ if $\NE=p^6$.}
 \end{cases} \]
 \end{enumromanii}
\end{enumalph}
\end{proposition}

\begin{proof} 
We recall from section \ref{sec:3point1} that $\rho_E = \Sp(\chi)$
and $\tau=\tau_E=\Sp(\chi)|_{I_p}$ for some quadratic character~$\chi$ 
of~$W_p$.

In part (a), we have $\NE = \cond(\tau) = p$ and by the conductor formula~\eqref{E:condSS} it follows that $\chi$ is unramified; in this case, 
\begin{equation} 
\tau = (\chi \otimes \Sp(1))|_{I_p} = \chi|_{I_p} \otimes \Sp(1)|_{I_p} = \Sp(1)|_{I_p} = \tauspnew{p}. 
\end{equation}
We turn to part (b).  We have $p^2 \mid \NE$, and formula~\eqref{E:condSS} gives $\NE = \cond(\tau) = p^{2m}$, 
so $\chi$ is ramified with $\cond(\chi) = p^m$.  
If $p \geq 3$, then any quadratic character $\chi\colon W_p \to \C^\times$ has conductor~$p$ and satisfies $\chi|_{I_p} = \epschar{p}$. 
Thus $\NE=p^2$, and $\tau \simeq \tauspnew{p} \otimes \epschar{p}$, proving~(i).  Otherwise, we have $p = 2$, and we conclude similarly, as in Lemma~\ref{L:e=2}: we have $\chi|_{I_p} = \epschar{-4},\epschar{\pm 8}$ of conductors $2^2,2^3$, proving~(ii).
\end{proof}

\subsection{Inertial types for \texorpdfstring{$p \geq 5$}{ellgeq5}} \label{sec:inert5}

The preceding section treated all cases of potentially multiplicative reduction, so for the rest of this paper we turn to the case of potentially good reduction. In particular, $N=0$. 
Here  we treat the case $p \geq 5$.

\begin{proposition} \label{P:peq5} 
Let $p \geq 5$.  Let $E$ be an elliptic curve over $\Q_p$ with additive potentially good reduction, semistability defect $e \geq 3$, and inertial type $\tau$.  Then the following statements hold.
\begin{enumalph}
 \item If $e \mid (p-1)$, then $\tau \simeq \tauPSnew{p}{1,1,e}$ is principal series.
 \item If $e \mid (p+1)$, then $\tau \simeq \tauSCnew{p}{u,2,e}$ is supercuspidal.
\end{enumalph}
\end{proposition}
\begin{proof}
Lemma~\ref{L:condBound} implies that $\tau$ has conductor $p^2$, and Lemma \ref{L:phi} shows that 
the image of~$\tau$ is cyclic of order $e=3,4,6$. 
From the classification in Proposition \ref{prop:non-prim-inert-type},
$\tau$ is reducible with finite image, hence it is either principal series or nonexceptional supercuspidal induced from the unramified quadratic extension $K=\Q_{p^2}$ of~$\Q_p$. 

Suppose that $\tau$ is principal series. Then, $\tau = \chi|_{I_p} \oplus \chi^{-1}|_{I_p}$, where $\chi$ is a character of $W_p$ of conductor $p$ and order $e$. To ease notation we write~$\chi$ also for~$\chi\spArt$. Thus, $\chi|_{I_p}$ factors through $(\Z_p/p\Z_p)^\times \simeq \F_{p}^\times$ a cyclic group of order $p-1$, so $e \mid (p-1)$.  Let $u \in \Z_p^\times$ and~$\chi_e$ be as in section~\ref{sec:notation}. So the reduction of~$u$ generates~$\F_p^\times$.  We have $\chi(u) = \zeta_e^c=\exp(2\pi ic/e)$ with $\gcd(c,e)=1$.  Since $e=3,4,6$, we must have $c \equiv \pm 1 \pmod{e}$, so $\chi|_{I_p}=\chi_e^{\pm 1}$ and either choice gives $\tau \simeq \tauPSnew{p}{1,1,e}$.

To finish, suppose $\tau$ is supercuspidal, obtained by induction of a character~$\chi$ of~$W_K$ of order~$e$ on~$I_K$. Since $\tau$ has conductor $p^2\calO_K$, it follows that $\chi\spArt$ viewed as a character of $K^\times$ has conductor~$p\calO_K$. Moreover, by Lemma~\ref{lem:centralChar}, we have $\tau = \chi|_{I_K} \oplus \chi^{-1}|_{I_K}$ and~$\chi\spArt|_{\Z_p^\times} = 1$ (as~$K$ is unramified over $\Q_p$).   
Now $\chi|_{I_K}$ factors through $(\Z_{p^2}/p\Z_{p^2})^\times \simeq \F_{p^2}^\times$ so $e \mid (p^2-1)=(p+1)(p-1)$.  Let $u \in \Z_{p^2}^\times$ be as in section \ref{sec:notation}, so its reduction generates $\F_{p^2}^\times$.  Then $u^{p+1}$ generates $(\Z_p/p\Z_p)^\times \leq (\Z_{p^2}/p\Z_{p^2})^\times$.  The condition $\chi|_{\Z_p^\times}=1$ implies that $e \mid (p+1)$.  We again have $\chi(u) = \zeta_e^c$ and as in the previous paragraph we must have $\chi|_{I_K}=\chi_e^{\pm 1}$
and either choice gives $\tau \simeq \tauSCnew{p}{u,2,e}$.
\end{proof}

\section{Inertial types for \texorpdfstring{$p = 3$}{elleq3}} \label{sec:inert3}

In this section, we treat the case $p=3$; see~Breuil--Conrad--Diamond--Taylor \cite{bcdt01} for the application of these types to the modularity of elliptic curves.

\subsection{Setup}

Throughout this section, we let ${\Mquad} \colonequals \Q_3(\sqrt{d})$ where $d= \pm 1,\pm 3$: when $d=1$ we have $\Mquad=\Q_3$.  Let $\calO_{\Mquad}$ be the valuation ring of ${\Mquad}$ and $\fp$ its maximal ideal.  When $\Mquad \neq \Q_3$, let $\chichar{d} \colon W_3 \to \C^\times$ be the quadratic character associated to ${\Mquad}$ and let $s \in W_3$ be a lift of the nontrivial element of $\Gal({\Mquad}\,|\,\Q_3)$. Recall (section \ref{sec:notation}) that $\epschar{-1} = \chichar{-1}|_{I_3}$ is trivial (the extension $\Q_3(\sqrt{-1}) \supseteq \Q_3$ is unramified) 
and $\epschar{3}=\chichar{\pm 3}|_{I_3}$ is the unique nontrivial quadratic character of $I_3$.  We have $\cond(\epschar{3})=3$.  For $d=-3$, let $\xi_6 \colonequals (1+\sqrt{-3})/2 \in K$ (to help distinguish it from roots of unity in $\C^\times$).

We begin with basic structural results concerning unit groups.  Let $\frakf = \frakp^f$ with $f\ge 1$, and let $(q) = \Z \cap \frakf$ with $q>0$.  Write 
\[ \Nm=\Nm_{K|\Q_3}\colon (\calO_K/\frakf)^\times \to (\Z_3/q)^\times \]
for the norm map (the identity when $K=\Q_3$).  Let 
\begin{equation}
U_{\frakf} \colonequals 
\begin{cases}
\{1\}, &\text{if}\, K = \Q_3;\\
\Nm((\calO_{\Mquad}/\frakf)^\times), & \text{otherwise}.
\end{cases}
\end{equation}

\begin{lemma}\label{lem:grp-structures-3} 
Table~\textup{\ref{table:U-norm-map-pr-3}} 
gives the structure and explicit generators for the groups $U_{\frakf}$ and $(\calO_{\Mquad}/\frakf)^\times/U_{\frakf}$, respectively.
\end{lemma}

\begin{table}
\normalfont\small\renewcommand{\arraystretch}{1.25}
\begin{tabular}{ >{$}c<{$} | >{$}c<{$} ||>{$}c<{$} | >{$}c<{$}|| >{$}r<{$} >{$}l<{$}  | >{$}r<{$} >{$}l<{$} }
{\Mquad} & d & \mathfrak{f} & f & \multicolumn{2}{c|}{$U_{\frakf}$} & 
\multicolumn{2}{c}{$(\calO_{\Mquad}/\frakf)^\times/U_{\frakf}$}   \\\hline\hline
\Q_3 & 1 & (3)^f & \geq 1 & \multicolumn{2}{c|}{$-$}  
& \langle -4 \rangle \!\!\!\!\!&\simeq \Z/(2\cdot 3^{f-1})\\\hline 
\Q_3(\sqrt{-1}) & -1 & (3)^f & \ge 1 & \langle -4 \rangle \!\!\!\!\!&\simeq \Z/(2\cdot3^{f-1}) 
& \langle \sqrt{-1} + 2\rangle \!\!\!\!\!&\simeq \Z/(4\cdot 3^{f-1}) \\\hline
\Q_3(\sqrt{3}) & 3 & \fp^f & \ge 1 & \langle 4 \rangle \!\!\!\!\!&\simeq \Z/(3^{\lfloor (f - 1)/2 \rfloor}) 
& \langle \sqrt{3} - 1\rangle \!\!\!\!\!&\simeq \Z/(2 \cdot 3^{\lfloor f/2 \rfloor}) \\\hline
\multirow{2}{*}{$\Q_3(\sqrt{-3})$} & \multirow{2}{*}{$-3$} & \multirow{2}{*}{$\fp^f$} & 1,2,3 & \langle 4 \rangle \!\!\!\!\!&\simeq \Z/(3^{\lfloor f/2 \rfloor}) & 
\langle \xi_6 - 3 \rangle \!\!\!\!\!&\simeq \Z/(2\cdot 3^{\lfloor f/2 \rfloor}) \\\cline{4-8}
& &  & \geq 4 & \langle 4 \rangle \!\!\!\!\!&\simeq \Z/(3^{\lfloor (f-1)/2 \rfloor}) & 
\langle -\xi_6 \rangle \times \langle \xi_6 - 3 \rangle \!\!\!\!\!&\simeq \Z/3 \times \Z/(2\cdot 3^{\lfloor (f-2)/2 \rfloor})\\
\end{tabular}
\caption{\rule{0pt}{2.5ex}Group structure of $U_{\frakf}$ and $(\calO_{\Mquad}/\frakf)^\times/U_{\frakf}$}
\label{table:U-norm-map-pr-3}
\end{table}

\begin{proof}
Following Cohen \cite[Proof of Theorem 4.2.10]{coh00}, this can be done with a direct calculation by induction on $f$; the code is online \cite{ourcode}.
\end{proof}

\begin{remark}\label{rem:projection:3}From the last column of Table~\ref{table:U-norm-map-pr-3}, we see that
for $K = \Q_3$, $\Q_3(\sqrt{-1})$ or $\Q_3(\sqrt{3})$, the quotient $(\calO_{\Mquad}/\frakf)^\times/U_{\frakf}$
is a cyclic group. So the canonical projection 
$$(\calO_{\Mquad}/\frakf)^\times/U_{\frakf} \to (\calO_{\Mquad}/\frakf')^\times/U_{\frakf'},$$
when $\frakf'\mid \frakf$, behave as expected. For $K=\Q_3(\sqrt{-3})$,
however, this is not quite the case. Indeed, in that case, the quotient $(\calO_{\Mquad}/\frakf)^\times/U_{\frakf}$
is no longer cyclic for $\frakf$ large enough. In particular, when $\frakf = \fp^4$ and $\frakf' = \fp^3$, the projection
$$\pi : (\calO_{\Mquad}/\frakf)^\times/U_{\frakf} \to (\calO_{\Mquad}/\frakf')^\times/U_{\frakf'}$$
sends $-\xi_6 \mapsto (\xi_6-3)^4 \pmod{\frakf'}$ and  $\xi_6 - 3 \mapsto \xi_6-3 \pmod{\frakf'}$. So $\pi(-\xi_6) \ne 0$, as
one might be tempted to think from the description of the groups involved.
\end{remark}

%

\begin{corollary}\label{cor:epsilon_d}
Let $\chi \colon W_3 \to \C^\times$ be a nontrivial quadratic character.  Then $\chi\spArt(4)=1$, and $\chi=\varepsilon_{-1}$ if and only if $\chi\spArt(-1)=1$.  In particular, $\chi$ is ramified if and only if $\chi\spArt(-1)=-1$.  
\end{corollary}

\begin{proof} 
By the first row of Table \ref{table:U-norm-map-pr-3}, for $\chi$ quadratic we must have $\chi\spArt(4)=1$ and so by local class field theory we have $\chi\spArt(-1)=1$ if and only if $\chi$ is unramified.
\end{proof}

\begin{corollary}\label{cor:onlyodd}
Let $\chi \colon W_K \to \C^\times$ be a character with $K$ ramified and $U_\frakf \leq \ker \chi\spArt$.  Then $\condexp(\chi)$ is even.
\end{corollary}

\begin{proof}
The bottom two rows of Table~\ref{table:U-norm-map-pr-3} show that if $f=\condexp(\chi)$ is odd, then by the floor function in fact $\chi$ factors through $(\calO_K/\frakp^{f-1})^\times$.
\end{proof}





\subsection{Result}

Returning to the classification in Proposition~\ref{prop:non-prim-inert-type}, and recalling that we have already treated the case of potentially multiplicative reduction in Proposition~\ref{prop:Emult} and that the exceptional case only happens for $p=2$, in this section we classify the inertial types for elliptic curves over $\Q_3$ with potentially good reduction, giving principal series ($d=1$) or (nonexceptional) supercuspidal ($d=-1,\pm 3$) and associated to a character as in \eqref{eqn:psnot} and \eqref{eqn:scnot}, respectively.

Recall our attempt at uniform and informative notation: we write our character of $W_K$ as $\chi_{(d,f,r)}$, where~$K$ has discriminant~$d$, with $r$ the order when restricted to $I_K$ and $f$ the conductor exponent.  
We work explicitly with such characters, taking values in $\C^\times$ in terms of roots of unity denoted $\zeta_k \colonequals \exp(2\pi i/k)$ for $k \geq 1$.  For the purposes of the restriction of inertia, we only care about $\chi_{(d,f,r)}|_{I_K}$, so to spell out the type we will describe $\chi_{(d,f,r)}\spArt|_{\calO_K^\times}$.  By definition of conductor exponent, this restriction then factors through $(\calO_K/\frakf)^\times$ where $\frakf=\frakp^f$, so can be given on generators as in the previous section.  
\begin{itemize}
\item In the principal series case, $U_\frakf$ is trivial by definition. 
\item In the supercuspidal case, we are in the situation of Lemma \ref{lem:centralChar}, which reads
\begin{equation}
\chi_{(d,f,r)}\spArt|_{\Z_3^\times} = \varepsilon_d.  
\end{equation} 
So by local class field theory, $U_\frakf \leq \ker \chi_{(d,f,r)}\spArt$. 
\end{itemize}
In either case, we may specify the character by its values on the generators of $(\calO_{\Mquad}/\frakf)^\times/U_{\frakf}$ given in Table \ref{table:U-norm-map-pr-3}.  In all cases but the last row of this table, the character is uniquely determined by the data $(d,f,r)$ (up to Galois conjugation), sending the designated generator to $\zeta_r$.  For the last row, we will write $\chi_{(-3,4,6),j}$ to be the character which takes the values $\chi(-\xi_6)=\zeta_3^j$ and $\chi(\xi_6-3)=-\zeta_3^{j-1}$ for $j=0,1,2$.  We summarize this data in Table~\ref{table:Q3}.  

\begin{table}
\normalfont\small\renewcommand{\arraystretch}{1.25}
\begin{tabular}{ >{$}c<{$} | >{$}c<{$} || >{$}c<{$} | >{$}c<{$} || >{$}c<{$} | >{$}c<{$} || >{$}c<{$} | >{$}c<{$} }
{\Mquad} & d & \mathfrak{f} & f & r & \text{values of~$\chi$ on generators} & \tau & \condexp(\tau) \\
\hline\hline
\Q_3 & 1 & 3^2 & 2 & 3 & \zeta_3 &\tauPSnew{3}{1, 2, 3} & 4 \\ 
\hline
\multirow{2}{*}{$\Q_3(\sqrt{-1})$} & \multirow{2}{*}{$-1$} & (3) & 1 & 4 & \zeta_4 & \tauSCnew{3}{-1,1,4} & 2 \\
\cline{3-8} 
 & & (3)^2 & 2 & 3 & \zeta_3 & \tauSCnew{3}{-1,2,3} & 4 \\
\hline
\Q_3(\sqrt{3}) & 3 & \frakp^2 & 2 & 6 & \zeta_6 & \tauSCnew{3}{3,2,6} & 3 \\
\hline
\multirow{2}{*}{$\Q_3(\sqrt{-3})$} & \multirow{2}{*}{$-3$} & \frakp^2 & 2 & 6 & \zeta_6 & \tauSCnew{3}{-3,2,6} & 3 \\
\cline{3-8}
& & \frakp^4 & 4 & 6 & \zeta_3^j,\, -\zeta_3^{j-1}  & \tauSCnew{3}{-3,4,6}_j\ (j=0,1,2) & 5 \\
\end{tabular}
\caption{\rule{0pt}{2.5ex}Nonspecial inertial types for $\Q_3$}
\label{table:Q3}
\end{table}

\begin{proposition} \label{P:elleq3prop}
Let $E$ be an elliptic curve over $\Q_3$ of conductor $\NE$ and inertial type $\tau_E$.  Suppose that $E$ has additive, potentially good reduction and semistability defect $e \geq 3$.  Then $\tau_E$ is given by one of the following cases:
\begin{enumalph}
  \item If $\NE = 3^2$, then $e=4$ and $\tau_E \simeq  \tauSCnew{3}{-1,1,4}$.
  \item If $\NE = 3^3$, then $e=12$ and $\tau_E \simeq \tauSCnew{3}{3,2,6}$ or $\tau_E \simeq \tauSCnew{3}{-3,2,6}$.
  \item If $\NE = 3^4$, then $e=3,6$, and:
  \begin{enumromanii}
  \item If $e=3$, then $\tau_E \simeq  \tauPSnew{3}{1,2,3}$ or $\tau_E \simeq \tauSCnew{3}{-1,2,3} $;  
  \item If $e=6$, then $\tau_E \simeq  \tauPSnew{3}{1,2,3} \otimes \epschar{3}$ or $\tau_E \simeq \tauSCnew{3}{-1,2,3}  \otimes \epschar{3}$.
  \end{enumromanii}
  \item If $\NE = 3^5$, then $e=12$ and $\tau_E \simeq   \tauSCnew{3}{-3,4,6}_j$ with $j=0,1,2$.
\end{enumalph}
\end{proposition}

\begin{proof}
We drop subscripts, writing e.g.\ $\tau=\tau_E$.  We have $N_E = 3^k$ with $2 \leq k=\condexp(\tau) \leq 5$ by Lemma~\ref{L:condBound}. 

We begin with (a), and suppose $k=2$.  Then we are in the case of tame reduction, and $\Phi$ is cyclic of order $e=4$ by Lemma~\ref{L:phi}.  We cannot have $\tau$ principal series, since then $\tau \simeq \chi|_{I_3} \oplus \chi^{-1}|_{I_3}$  with $\chi$ a character of $W_3$ of conductor exponent $1$ and order $4$, which does not exist.  So $\tau$ must be (nonexceptional) supercuspidal.  Since $e = 4 \mid (3 + 1)$, we conclude that $\tau \simeq  \tauSCnew{3}{-1,1,4}$ by a similar argument as in the proof of Proposition~\ref{P:peq5}(b).   

Next, we prove part (b) and (d) and suppose $k=3$ or $k=5$.   Since $k>1$ is odd, by \eqref{E:condBC}, $\tau$ is obtained by induction of a character $\chi$ from a ramified quadratic extension ${\Mquad} \supset \Q_3$ with $d=\pm 3$ and $m=\condexp(\chi) + v_3(d)=\condexp(\chi)+1$.  In particular, we are in case (iv) of Proposition~\ref{prop:non-prim-inert-type}: $\tau$ is irreducible and in this case it is necessary and sufficient for $(\chi|_{I_K})\spArt$ not to factor via the norm.  Hence, $e=12$ by
Corollary~\ref{cor:notfactorbyinertia}.
This gives three cases:
\begin{itemize} 
\item Suppose $k=3$, so $\condexp(\chi)=2$.  By Lemma~\ref{lem:grp-structures-3}, we have 
$$(\calO_{\Mquad}/\fp^2)^\times/U_{\frakf}  = \langle u \rangle \simeq \Z/6,$$ 
where $u = \sqrt{3} - 1$ for $d=3$ and $u=\xi_6 - 3$ for $d=-3$.  
Since $\chi$ is primitive with $\condexp(\chi)=2$, we must have $\chi(u)=\pm\zeta_3^j$ with~$j=1, 2$.
Applying Corollary~\ref{cor:epsilon_d} to $\chi|_{\Z_3^\times} =\varepsilon_{d}$ we conclude that $\chi(-1)=-1$ and hence $\chi(u)=-\zeta_3^j$ with $j=1$ or $j=2$, which are Galois conjugate.  Note that $\chi$ does not factor via the norm by Corollary~\ref{cor:factorvianorm}(a), so the type indeed occurs.  We obtain the same induction from either character, so we may take $\chi|_{I_{\Mquad}}=\chi_{(d, 2, 6)}|_{I_{\Mquad}}$, giving $\tau \simeq \tauSCnew{3}{d,2,6}$.

\item Next, suppose that $k=5$ and $d = 3$.  By Lemma~\ref{lem:grp-structures-3}, $\chi|_{I_K}$ is a nontrivial character on $$(\calO_{\Mquad}/\fp^4)^\times/U_{\frakf}  = \langle u \rangle \simeq \Z/18.$$
The order of $\chi(u)$ is not a divisor of~$6$, otherwise $\chi$ would be imprimitive (i.e., $\condexp(\chi) \leq 3$); thus
$\chi|_{I_K}$ has order $9$~or~$18$ and so $9 \mid e=12$, a contradiction.  Therefore this case does not occur.

\item Finally, suppose $k=5$ and $d=-3$.  Now by Lemma~\ref{lem:grp-structures-3} we have a character on
$$(\calO_{\Mquad}/\fp^4)^\times/U_{\frakf}  =  \langle u_1 \rangle \times \langle u_2 \rangle \simeq \Z/3 \times \Z/6$$  
where $u_1 = -\xi_6$ and $u_2 = \xi_6 - 3$.
The primitivity condition here is more subtle. Indeed, if $\chi$ is imprimitive then $\chi = \theta \circ \pi$ for some character~$\theta$
with $\condexp(\theta) \leq 3$, where $\pi$ is the projection defined in Remark~\ref{rem:projection:3}.
Since $(\calO_{\Mquad}/\fp^3)^\times/U_{\frakf} \simeq \Z/6$
we conclude that
$\chi(u_1 u_2^2) = \theta (u_2^6) = 1$. Thus $\chi$ is primitive if and only if $\chi(u_1) \neq \chi(u_2)^{-2}$.
Furthermore, Corollary~\ref{cor:epsilon_d} applied to $\chi|_{\Z_3^\times} =\varepsilon_{d}$ gives $\chi(-1) = \chi(u_2^3) = -1$ and so $\chi|_{I_K}$ does not factor through the norm by Corollary~\ref{cor:factorvianorm}(a). We conclude that $\chi(u_2) = \zeta_6^i = (-1)^i\zeta_3^{2i}$ with $i=1,3,5$ and the constraints on $\chi(u_1)$ follow from the primitivity condition.
More precisely, up to replacing $\chi$ by its Galois conjugate (which coincides with $\chi^{-1}$ on inertia), there are three possible characters defined by
\[
 \chi(u_1) = \zeta_3^j \quad \text{ and } \quad \chi(u_2) = -\zeta_3^{j-1} \quad \text{ for } j = 0, 1, 2.
\]
Therefore, we can choose $\chi = \chi_{(-3, 4, 6)_j}$  for $j=0,1,2$. Thus $\tau \simeq   \tauSCnew{3}{-3,4,6}_j$.
\end{itemize}

We conclude by proving (c).  Suppose $k=4$.  
\begin{itemize}
\item First, suppose that $\tau$ is principal series. Then $\tau = \chi|_{I_3} \oplus \chi^{-1}|_{I_3}$, where $\chi$ is a character of $W_3$ with conductor exponent $2$. 
By Lemma~\ref{lem:grp-structures-3}, we have that $\chi|_{I_3}$ factors through 
$$(\Z_3/(3^2))^\times \simeq \langle -4\rangle \simeq \Z/6.$$
Since~$\chi$ is primitive with $\condexp(\chi)=2$, we have $\chi(4)=\zeta_3^{\pm 1}$. Twisting by $\epschar{3}$, we can assume that $\chi(-1)=1$. 
Thus, there are two possibilities for $\chi|_{I_3}$, which are $\chi_{(1,2,3)}$ and $\chi_{(1,2,3)}^{-1}$. Thus $\tau \simeq  \tauPSnew{3}{1,2,3}$ (hence $e=3$)
or $\tau \simeq  \tauPSnew{3}{1,2,3} \otimes \epschar{3}$ (hence $e=6$).
\item We are left with the case where $\tau$ is supercuspidal.  We first claim that $K$ is unramified: indeed, if $K$ is ramified, then again from \eqref{E:condBC} we have $4=\condexp(\chi)+\condexp(\psi_K)=\condexp(\chi)+1$ so $\condexp(\chi)=3$.  But this contradicts Corollary \ref{cor:onlyodd}.  Thus $K$ is unramified, so $d=-1$, and $\condexp(\chi)=2$ by \eqref{E:condBC}.  Reading Lemma \ref{lem:grp-structures-3} one last time, we have $\chi|_{I_K}$ on
$$\left(\calO_{{\Mquad}}/(3^2)\right)^\times/U_{\frakf} = \langle u \rangle \simeq \Z/12 $$
where $u = \sqrt{-1} + 2$. By Corollary~\ref{cor:epsilon_d}, $\chi(4) = \chi(-1) = 1$.  By primitivity, $\chi(u)$ must have order $3$, $6$, or $12$. However, by Lemma~\ref{L:phi}, the case $e=12$ only arises for nonabelian inertia, which is not possible as $K \supset \Q_3$ is unramified. 
Thus $\chi(u) = \zeta_6^j$, with $j \in \{1,2,4,5\}$, which give two pairs of conjugate characters.  Note that Corollary~\ref{cor:factorvianorm}(b) implies that indeed $\chi$ does not factor through the norm. 

To finish, we recognize these two as twists.  We note that $\chi' = \chi \cdot (\varepsilon_{3}|_{W_K})$ is also a character of~$W_K$ with conductor exponent~$2$ with $\delta(g)$ having order~$3$ or~$6$, and moreover
\[
 \delta(u) = \chi(u)\chi_3(\Nm_{K|\Q_3}(u)) = \chi(u)\eps_3(u^2) = \chi(u) = 1 \]
for all $u \in \Z_3^\times$.  Therefore, $\chi'|_{\calO_K^\times}$ must be one of the previous four characters.
Thus, up to twisting by $\epschar{3}$, we can assume that $\chi(g) = \zeta_6^j$, for $j=2,4$, which is the same as requiring that $\chi(g) = \zeta_3^{\pm 1}$ which are now conjugate.  Thus we can take $\chi|_{I_K} = \chi_{(-1,2,3)}$, and either $\tau \simeq \tauSCnew{3}{-1,2,3}$ with $e=3$
or $\tau \simeq \tauSCnew{3}{-1,2,3} \otimes \varepsilon_3$ with $e=6$.
\end{itemize}

The proof is then complete by exhaustion of cases.
\end{proof}

\subsection{Explicit realization} \label{sec:exm3}

In Table~\ref{table:types-Q3}, we give an example of a curve realizing each inertial type we computed over $\Q_3$ for the case of potentially good reduction; in particular those described in Proposition~\ref{P:elleq3prop} and Table~\ref{table:Q3}.  Additional columns are explained below.  

\begin{table}
\normalfont\small\renewcommand{\arraystretch}{1.25}
\begin{tabular}{ >{$}c<{$} || >{$}c<{$} | >{$}c<{$} | >{$}c<{$} | >{$}c<{$} | >{$}c<{$} || >{$}c<{$} }
\tau & e & \condexp(\tau) & \Nm_{L|K}(\calO_L^\times)/U_\frakf & L' & \Gal(L\,|\,\Q_3) & E \\
\hline\hline
\text{trivial} &  1 & 0 & - & \LMFDBL{3.1.0.1} & \LMFDBG{1T1} \simeq C_1 & \LMFDBE{11a1} \\
\hline
\epschar{3} &  2  & 2 & - & \LMFDBL{3.1.0.1} & \LMFDBG{2T1} \simeq C_2 & \LMFDBE{99d2} \\
\hline
\tauPSnew{3}{1, 2, 3} & 3 & 4 & - & \LMFDBL{3.3.4.2} & \LMFDBG{3T1} \simeq C_3 & \LMFDBE{162b1} \\
\hline
\tauSCnew{3}{-1, 2, 3} &  3 & 4 & \langle 3 \rangle \simeq \Z/4 & \LMFDBL{3.3.4.4} & \LMFDBG{3T2} \simeq S_3 & \LMFDBE{162d1} \\
\hline
\tauSCnew{3}{-1,1,4} &  4 & 2 & \textup{trivial} & \LMFDBL{3.4.3.1} & \LMFDBG{4T3} \simeq D_4 & \LMFDBE{36a1} \\
\hline
\tauPSnew{3}{1, 2, 3} \otimes \epschar{3} &  6 & 4 & -& \LMFDBL{3.6.9.11} & \LMFDBG{6T1} \simeq C_6 & \LMFDBE{162c2} \\
\hline
\tauSCnew{3}{-1, 2, 3} \otimes \epschar{3} &  6  & 4 & \langle 6 \rangle \simeq \Z/2 & \LMFDBG{3.6.9.12} & \LMFDBG{6T3} \simeq D_6 & \LMFDBE{162a2} \\
\hline
\tauSCnew{3}{3,2,6} &  12 & 3 & \textup{trivial}  & \LMFDBL{3.12.15.1} & \LMFDBG{12T15} \simeq 2\cdot D_6 & \LMFDBE{27a1} \\
\hline
\tauSCnew{3}{-3,2,6} &  12 & 3 & \textup{trivial}  & \LMFDBL{3.12.15.12} & \LMFDBG{12T13} \simeq 2\cdot D_6  & \LMFDBE{54a1} \\
\hline
\tauSCnew{3}{-3,4,6}_0 &  12  & 5 & \langle (1,0) \rangle \simeq \Z/3  & \LMFDBL{3.12.23.122} & \LMFDBG{12T13} \simeq 2\cdot D_6 & \LMFDBE{972b1} \\
\hline
\tauSCnew{3}{-3,4,6}_1 &  12 & 5 & \langle (0, 2)\rangle\simeq \Z/3 & \LMFDBL{3.12.23.20} & \LMFDBG{12T13} \simeq 2\cdot D_6 & \LMFDBE{243b1} \\
\hline
\tauSCnew{3}{-3,4,6}_2 &  12 & 5 & \langle (1, 4)\rangle \simeq \Z/3 & \LMFDBL{3.12.23.14} & \LMFDBG{12T13} \simeq 2\cdot D_6 & \LMFDBE{243a1} 
\end{tabular}
\vspace{0.2em}
\caption{\rule{0pt}{2.5ex}Types, defining fields, and elliptic curves realizing each inertial type over $\Q_3$ in the case of potentially good reduction}
\label{table:types-Q3}
\end{table}

To this end, let $\tau$ be an inertial type.  Let $L' \supseteq \Q_p$ be an extension of minimal degree such that $\Q\spun_p L'$ is the extension of~$\Q\spun_p$ cut out by $\tau$.  If $[L':\Q_p] = [\Q\spun_p L' : \Q\spun_p]$, then we call $L'$ a \defi{descent} of the inertial field.  Let $L$ be the normal closure of $L'$.  When $\tau$ is a non-exceptional supercuspidal type, there is a quadratic subfield $K \subset L$ and a character $\chi\colon K^\times \to \C^\times$ such that,
by class field theory,
\begin{equation}
\ker \chi^A = \Nm_{L_\tau|K}(L_\tau^\times) \subset K^\times
\end{equation}
and the restriction to inertia
$\chi|_{I_K}$ satisfies
\begin{equation}
\ker (\chi^A|_{(\calO_K/\frakf)^\times/U_\frakf}) = \Nm_{L_\tau|K}(\calO_{L_\tau}^\times)/U_\frakf \hookrightarrow (\calO_K/\frakf)^\times/U_\frakf.
\end{equation}

\begin{proposition} \label{prop:3eachfield}
For $p=3$ and each inertial type $\tau$ arising from an elliptic curve with additive, potentially good reduction, there is a unique descent $L'$ of the inertial field.  Moreover, either $L'=L$ is Galois or the compositum $L=\Q_9 L'$ is Galois, where $\Q_9$ is the quadratic unramified extension of $\Q_3$.
\end{proposition}

\begin{proof}
There are $12$ such types; we can exclude the trivial case leaving $11$.  Accompanying code is available online \cite{ourcode}.

By search, we find a list of $11$ elliptic curves $E$ over $\Q_3$ and a list of $11$ nonisomorphic totally ramified extensions $L'$ such that for each elliptic curve $E$ there is a unique field $L'$ of minimal degree such that $E_{L'}$ has good reduction.  (Indeed, most curves obtain good reduction over exactly one listed field; some obtain good reduction over two or more fields but only one with minimal degree.)

We verify the second statement, that the fields are either Galois over $\Q_3$ or become so after taking the compositum with the quadratic unramified extension; we then verify that the extensions are pairwise nonisomorphic as extensions of $\Q_9$ hence as extensions of $\Q\spun_3$.  

The fields $L'$ are listed in Table~\ref{table:types-Q3} by LMFDB label.  Moreover, for each nonabelian Galois closure $L \supseteq \Q_3$ of a $L'$ in our list,
we computed the image of $\Nm_{L|K}(\calO_{L}^\times)/U_\frakf$ inside $(\calO_K/\frakf)^\times/U_\frakf$ where the latter is given by the group structure and generators in Table~\ref{table:U-norm-map-pr-3}.

Given that we have a distinct list of fields and a list of types with the same cardinality,  the result follows.  
\end{proof}

To finish the proof of Table~\ref{table:types-Q3} we need to match the descent fields with the types.  Since the curve \LMFDBE{11a1} has good reduction at $3$ and its quadratic twist by~$3$ is the curve \LMFDBE{99d2}, the first two rows follow immediately.  The principal series types correspond by construction to abelian extensions~$L=L' \supseteq \Q_3$. Thus $\tauPSnew{3}{1,2,3}$ corresponds to the unique cyclic extension in the list with ramification $e=3$
and $\tauPSnew{3}{1,2,3} \otimes \epschar{3}$ to the unique cyclic extension with $e=6$.

We are then left with the supercuspidal types.  The fields $L'$ are totally ramified and their Galois closure is obtained by the compositum with at most an unramified quadratic extension, so the computed norm groups $\Nm_{L|K}(\calO_L^\times)/U_\frakf$ uniquely identify it as an extension of $K$ by local class field theory, and we may freely pass between $L$ and its normal closure.  To select the quadratic subextension $K \subseteq L$ and the conductor~$\frakf$ for the calculation of the norm groups, we used the following two facts:
\begin{enumerate}
\item The conductor of a type $\tau$ fixing $L_\tau$ matches the conductor of the curve that corresponds to~$L_\tau$. 
\item The description of the types in the list with a fixed conductor gives the candidates quadratic subextensions~$K \subset L$; for example, Proposition~\ref{P:elleq3prop}(d) implies that fields corresponding to curves of conductor $3^5$ are obtained from characters of $K =\Q_3(\sqrt{-3})$ (even though $\Q_3(\sqrt{3})$ and $\Q_3(\sqrt{-1})$ are also contained in those fields).
\end{enumerate}
Comparing the output of this calculation to the definition of the characters in Table~\ref{table:Q3} completes the proof of all the rows except for with $\condexp(\tau_E)=3$. This is because the fields corresponding to \LMFDBE{27a1} and \LMFDBE{54a1} both contain the two quadratic fields $\Q_3(\sqrt{\pm 3})$ and there is one type of conductor $3^3$ induced from each of these quadratic fields, namely $\tauSCnew{3}{\pm 3, 2,6}$.  So to finish, we observe that for $\tauSCnew{3}{\pm 3, 2,6}$ the
induced characters have order~$6$ and $(\calO_K/\frakf)^\times/U_\frakf \simeq \Z/6$, hence the norm group must be trivial 
as per the table. Further, this norm group can only occur if we are taking norms towards a quadratic field from which the type can be induced from. Thus if the trivial norm group occurs also for the other quadratic field then the type must be triply imprimitive by Proposition~\ref{prop:triply-imprimitive}.  But this is not the case because its Galois group
is isomorphic to \href{https://beta.lmfdb.org/Groups/Abstract/24.8}{\textsf{24.8}} of order $24$, a group which also goes by the names
\[ 2 \cdot D_6 \simeq (C_6 \times C_2) : C_2 \simeq C_3 : D_4; \] 
however, this group has center $C_2$ and the quotient modulo center has order $12$, hence the projectivization of its image cannot be isomorphic to $D_2 \simeq C_2 \times C_2$.

Having realized types explicitly this way, we obtain the following corollary.  

\begin{corollary} \label{cor:algQ3}
Let $E$ be an elliptic curve over $\Q_3$ with potentially good reduction.  Then there is a unique field $L'$ in Table~\ref{table:types-Q3} of minimal degree such that $E$ obtains good reduction over $L'$, and $\tau_E$ is given by the type that corresponds to this $L'$.
\end{corollary}

Code implementing Corollary~\ref{cor:algQ3} is available online \cite{ourcode}.

\section{Nonexceptional inertial types at \texorpdfstring{$p=2$}{elleq2}} \label{sec:inert2nonexcept}

In this section, we begin our consideration of inertial types for the case $p=2$.  We treat all inertial types but for the exceptional supercuspidal types, leaving the latter for section~\ref{S:exceptionalThm}.  The outline of the argument is the same as in the case $p=3$, just with more technical complications.

\subsection{Setup and statement of result}

In this section, we let ${\Mquad}=\Q_2(\sqrt{d})$ be one of the eight at most quadratic extensions of~$\Q_2$,  so $d=1,-4,5,\pm 8,-20,\pm 40$.
The unique nontrivial unramified extension has $d=5$; the remaining nontrivial extensions have conductor exponent $2$ or $3$.
As before, when $K \neq \Q_2$ let $s \in W_{2}$ be a lift of the non-trivial element of $\Gal({\Mquad}\,|\,\Q_2)$.

As in the case $p=3$, we will need to know the quadratic character $\epschar{d}$ on inertia explicitly, as follows.
\begin{itemize}
 \item We have $\epschar{-4}=\epschar{-20}$, with $\condexp(\epschar{-4})=2$ and its restriction $\epschar{-4}\spArt|_{\Z_2^\times}$ factors through $(\Z_2/2^2 \Z_2)^\times$ and
 \[
  \epschar{-4}(-1) = -1. 
 \]
 \item We have $\epschar{8}=\epschar{40}$, with $\condexp(\epschar{8})=3$ and similarly $\epschar{8}$ is defined on $(\Z_2/2^3\Z_2)^\times = \langle -1 \rangle \times \langle 5 \rangle \simeq \Z/2 \times \Z/2$ by
 \[
  \epschar{8}(-1) = 1, \quad \epschar{8}(5) = -1. 
 \]
 \item We have $\epschar{-8}=\epschar{8}\epschar{-1}=\epschar{-40}$, with $\condexp(\epschar{8})=3$ and 
 \[
  \epschar{-8}(-1) = -1, \quad \epschar{-8}(5) = -1. 
 \]
\end{itemize}

As in section~\ref{sec:inert3}, we begin by setting up the structures of finite quotients of unit groups.  We adopt the same notation: $\frakf = \frakp^f$ with $f\ge 1$ and $q = \Z_2 \cap \frakf$, with $\Nm=\Nm_{K|\Q_2} \colon (\calO_K/\frakf)^\times \to (\Z_2/q)^\times$ the norm map, and finally 
\begin{align*}
U_{\frakf} \colonequals 
\begin{cases}
\{1\}, &\text{if}\, K = \Q_2;\\
\Nm((\calO_{\Mquad}/\frakf)^\times), & \text{otherwise}.
\end{cases}
\end{align*}

\begin{lemma}\label{lem:grp-structures-pr-2}
Tables \textup{\ref{table:norm-grps-pr-2}}, \textup{\ref{table:U-quo-unit-grps-pr-2}}, and \textup{\ref{table:Z2-quo-unit-grps-pr-2}} give the structure and generators for the groups $U_{\frakf}$, $(\calO_K/\frakf)^\times/U_\frakf$, and $(\Z_2/q)^\times/U_\frakf$, respectively.
\end{lemma}

\begin{table}
\normalfont\small\renewcommand{\arraystretch}{1.25}
\begin{tabular}{ >{$}c<{$} | >{$}c<{$} || >{$}c<{$} | >{$}c<{$} || >{$}r<{$} >{$}l<{$} }
{\Mquad} & d & \mathfrak{f} & f & \multicolumn{2}{c}{$U_{\frakf}$} \\\hline\hline
\Q_2 & 1 & (2)^f & \geq 1 & \multicolumn{2}{c}{$-$} \\\hline
\multirow{3}{*}{$\Q_2(\sqrt{5})$} & \multirow{3}{*}{$5$} & (2) & 1 & \multicolumn{2}{c}{trivial} \\
 & & (2)^2 & 2 & \langle -1 \rangle \!\!\!\!\!&\simeq \Z/2 \\
 & & (2)^f & \geq 3 & \langle -1 \rangle \times \langle 3 \rangle \!\!\!\!\!&\simeq \Z/2 \times \Z/2^{f-2} \\ \hline

\multirow{2}{14ex}{\centering $\Q_2(\sqrt{m})$, \\ $m=-1,-5$} & \multirow{2}{*}{$-4,-20$} & \frakp^f & 1,2,3,4 & \multicolumn{2}{c}{trivial} \\
 & & \frakp^f & \geq 5 & \langle -3 \rangle \!\!\!\!\!&\simeq \Z/2^{\lfloor (f-3)/2 \rfloor} \\\hline

\multirow{3}{14ex}{\centering $\Q_2(\sqrt{m})$, \\ $m=2,10$} & \multirow{3}{*}{$8,40$} & \frakp^f & 1,2 & \multicolumn{2}{c}{trivial} \\
 & & \frakp^f & 3,4,5,6 & \langle -1 \rangle \!\!\!\!\!&\simeq \Z/2 \\
 & & \frakp^f & \geq 7 & \langle -1 \rangle \times \langle -9 \rangle \!\!\!\!\!&\simeq \Z/2 \times \Z/2^{\lfloor (f-5)/2 \rfloor} \\\hline

\multirow{3}{14ex}{\centering $\Q_2(\sqrt{m})$, \\ $m=-2,-10$} & \multirow{3}{*}{$-8,-40$} & \frakp^f & 1,2 & \multicolumn{2}{c}{trivial} \\
 & & \frakp^f & 3,4,5,6 & \langle 3 \rangle \!\!\!\!\!&\simeq \Z/2 \\
 & & \frakp^f & \geq 7 & \langle 3 \rangle \!\!\!\!\!&\simeq \Z/2 \times \Z/2^{\lfloor (f-3)/2 \rfloor} \\
\end{tabular}
\caption{\rule{0pt}{2.5ex}Group structure of $U_{\frakf}$}
\label{table:norm-grps-pr-2}
\end{table}

\begin{table}
\normalfont\small\renewcommand{\arraystretch}{1.25}
\begin{tabular}{ >{$}c<{$} | >{$}c<{$} || >{$}c<{$} | >{$}c<{$} || >{$}r<{$} >{$}l<{$} }
{\Mquad} & d & \mathfrak{f} & f & \multicolumn{2}{c}{$(\calO_K/\frakf)^\times/U_{\frakf}$} \\\hline\hline
\multirow{3}{*}{$\Q_2$} & \multirow{3}{*}{$1$} & (2) & 1 & \multicolumn{2}{c}{trivial} \\
& & (2)^2 & 2 & \langle -1 \rangle \!\!\!\!\!&\simeq \Z/2 \\
& & (2)^f & \geq 3 & \langle -1 \rangle \times \langle 5 \rangle \!\!\!\!\!&\simeq \Z/2 \times \Z/2^{f-2} \\
\hline

\multirow{2}{*}{$\Q_2(\sqrt{5})$} & \multirow{2}{*}{$5$} & \multirow{2}{*}{$(2)^f$} & 1,2 & \langle (-1+\sqrt{5})/2 \rangle \!\!\!\!\!&\simeq \Z/(3 \cdot 2^{f-1}) \\
 & & & \geq 3 & \langle \sqrt{5} \rangle \times \langle (-1+\sqrt{5})/2 \rangle \!\!\!\!\!&\simeq \Z/2 \times \Z/(3 \cdot 2^{f-2}) \\ \hline

\multirow{3}{14ex}{\centering $\Q_2(\sqrt{m})$, \\ $m=-1,-5$} & \multirow{3}{*}{$-4,-20$} & \frakp & 1 & \multicolumn{2}{c}{trivial} \\
 & & \frakp^2 & 2 & \langle \sqrt{m} \rangle \!\!\!\!\!&\simeq \Z/2 \\
 & & \frakp^f & \geq 3 & \langle \sqrt{m} \rangle \times \langle 2\sqrt{m}-1 \rangle \!\!\!\!\!&\simeq \Z/4 \times \Z/2^{\lfloor (f-2)/2 \rfloor} \\\hline

\multirow{2}{14ex}{\centering $\Q_2(\sqrt{m})$, \\ $m=\pm 2,\pm 10$} & \multirow{2}{*}{$\pm 8,\pm 40$} & \multirow{2}{*}{$\frakp^f$} & 1,2,3,4 & \langle \sqrt{m}-1 \rangle \!\!\!\!\!&\simeq \Z/2^{\lfloor f/2 \rfloor} \\
 & & & \geq 5 & \langle -3 \rangle \times \langle \sqrt{m}-1 \rangle \!\!\!\!\!&\simeq \Z/2 \times \Z/2^{\lfloor f/2 \rfloor}\\
\end{tabular}
\caption{\rule{0pt}{2.5ex}Group structure of $(\calO_{\Mquad}/\frakf)^\times/U_{\frakf}$}
\label{table:U-quo-unit-grps-pr-2}
\end{table}

\begin{table}
\normalfont\small\renewcommand{\arraystretch}{1.25}
\begin{tabular}{ >{$}c<{$} | >{$}c<{$} || >{$}c<{$} | >{$}c<{$} || >{$}r<{$} >{$}l<{$} }
{\Mquad} & d & \mathfrak{f} & f & \multicolumn{2}{c}{$(\Z_2/q)^\times/U_{\frakf}$} \\\hline\hline
\multirow{3}{*}{$\Q_2$} & \multirow{3}{*}{$1$} & (2) & 1 & \multicolumn{2}{c}{trivial} \\
& & (2)^2 & 2 & \langle -1 \rangle \!\!\!\!\!&\simeq \Z/2 \\
& & (2)^f & \geq 3 & \langle -1 \rangle \times \langle 5 \rangle \!\!\!\!\!&\simeq \Z/2 \times \Z/2^{f-2} \\
\hline

\Q_2(\sqrt{5}) & 5 & (2)^f & \geq 1 & \multicolumn{2}{c}{trivial} \\\hline

\multirow{2}{14ex}{\centering $\Q_2(\sqrt{m})$, \\ $m=-1,-5$} & \multirow{2}{*}{$-4,-20$} & \multirow{2}{*}{$\frakp^f$} & 1,2 & \multicolumn{2}{c}{trivial} \\
 & & & \geq 3 & \langle -1 \rangle \!\!\!\!\!&\simeq \Z/2 \\\hline

\multirow{2}{14ex}{\centering $\Q_2(\sqrt{m})$, \\ $m=2,10$} & \multirow{2}{*}{$8,40$} & \multirow{2}{*}{$\frakp^f$} & 1,2,3,4 & \multicolumn{2}{c}{trivial} \\
 & & & \geq 5 & \langle -3 \rangle \!\!\!\!\!&\simeq \Z/2 \\\hline

\multirow{2}{14ex}{\centering $\Q_2(\sqrt{m})$, \\ $m=-2,-10$} & \multirow{2}{*}{$-8,-40$} & \multirow{2}{*}{$\frakp^f$} & 1,2,3,4 & \multicolumn{2}{c}{trivial} \\
 & & & \geq 5 & \langle -1 \rangle \!\!\!\!\!&\simeq \Z/2\\

\end{tabular}
\caption{\rule{0pt}{2.5ex}Group structure of $(\Z_2/q)^\times/U_{\frakf}$}
\label{table:Z2-quo-unit-grps-pr-2}
\end{table}

\begin{proof}
We use the same kind of induction on $k$ as in the proof of \cite[Theorem 4.2.10]{coh00} and a case-by-case analysis.  We illustrate the calculation in \ref{table:U-quo-unit-grps-pr-2} by providing an example.  
Suppose $d=2,10$.  Since $K$ is ramified, then $q=2^{k}$ where $k \colonequals \lceil f/2 \rceil$.
\begin{itemize}
\item If $f=1,2$ then $k=1$, and both $U_\mathfrak{f}$ 
and~$(\Z/2^k)^\times$ are trivial.
\item If $f=3,4$ then $k=2$, and $U_\mathfrak{f} = \langle -1 \rangle$ and $(\Z/2^k)^\times = \langle -1 \rangle$. 
\item If $f=5,6$ then $k=3$, we have $U_\mathfrak{f} = \langle -1 \rangle$, $(\Z/2^k)^\times = \langle -1 \rangle \times \langle 5 \rangle \simeq \Z/2 \times \Z/2$. 
\item Finally, for $k \geq 7$, we have
\[ \left\lfloor \frac{f-5}{2}\right\rfloor = \left\lceil \frac{f}{2} \right\rceil - 2 = k-2. \]
Therefore, $U_\mathfrak{f}$ contains exactly half the elements of $(\Z_2/q)^\times$.
\end{itemize}
The remaining cases are similar.
\end{proof}

\begin{remark}\label{rem:projection_2}
Similarly to Remark~\ref{rem:projection:3}, the canonical projection maps between quotients behave as expected, except for the field $K=\Q_2(\sqrt{5})$.
In that case, the projection map
$$\pi : (\calO_{\Mquad}/\fp^3)^\times/U_{\frakf} \to (\calO_{\Mquad}/\fp^2)^\times/U_{\frakf}$$
is given by
\[ \pi(\sqrt{5}) = \left(\frac{-1 + \sqrt{5}}{2}\right)^3 \pmod{\fp^2}, \qquad
\pi\left(\frac{-1 + \sqrt{5}}{2}\right) = \frac{-1 + \sqrt{5}}{2} \pmod{\fp^2}.
\]
So $\pi(\sqrt{5}) \ne 0$ contrary to what one might think at first glance.
\end{remark}


\begin{corollary}\label{cor:epsilon-d-pr-2}
The following statements hold.
\begin{enumalph}
\item For $d=5$, we have 
$$\chi|_{\Z_2^\times} = \varepsilon_5 \iff \chi|_{U_\frakf} = 1 \iff \chi(3) = \chi(-1) = 1.$$
\item For $d=-4,-20$ and~$f \geq 3$, we have 
$$\chi|_{\Z_2^\times} = \varepsilon_{-4} \iff \chi|_{U_\frakf} = 1 \iff \chi(3) = \chi(-1) = -1.$$
\item For $d=8,40$ and~$f \geq 5$, we have 
$$\chi|_{\Z_2^\times} = \varepsilon_8 \iff \chi|_{U_\frakf} = 1\,\,\text{and}\,\,\chi(-3) = -1\,\,\iff \chi(3) = -1\,\,\text{and}\,\,\chi(-1) = 1.$$
\item For $d=-8,-40$ and~$f \geq 5$, we have 
$$\chi|_{\Z_2^\times} = \varepsilon_{-8} \iff \chi|_{U_\frakf} = 1\,\,\text{and}\,\,\chi(-3) = -1\,\,\iff \chi(3) = 1\,\,\text{and}\,\,\chi(-1) = -1.$$
\end{enumalph}
\end{corollary}

\begin{proof}
This follows from Lemma~\ref{lem:grp-structures-pr-2}.
\end{proof}

We now focus on the cases of principal series and nonexceptional supercuspidal, where we attached a character~$\chi_{(d,f,r)}$ of~$W_K$ where~$K$ has discriminant~$d$, whose conductor exponent is~$f$ and whose order on~$I_K$ is~$r$.  As in the case $p=3$, 
to define $\chi_{(d,f,r)}|_{I_K}$ satisfying $\chi_{(d,f,r)}|_{\Z_2^\times} = \varepsilon_d$, it is enough to give the values of~$\chi_{(d,f,r)}$ on generators of~ $(\calO_{\Mquad}/\frakf)^\times/U_{\frakf}$.

In preparation for our final result, we list the relevant characters in this way in Table~\ref{table:type-Q_2}. 

\begin{table}
\normalfont\small\renewcommand{\arraystretch}{1.25}
\begin{tabular}{ >{$}c<{$} | >{$}c<{$} || >{$}c<{$}| >{$}c<{$} || >{$}c<{$} | >{$}c<{$} ||  >{$}c<{$} | >{$}c<{$} }
{\Mquad} & d & \mathfrak{f} & f & r & \textup{values of~$\chi$ on generators} & \tau & \condexp(\tau) \\
\hline\hline
\Q_2 & 1 & (2)^4 & 4 & 4 & 1,\, i &\tauPSnew{2}{1,4,4} & 8 \\
\hline
\multirow{2}{*}{$\Q_2(\sqrt{5})$}& \multirow{2}{*}{$5$} & (2) & 1 & 3 & \zeta_3 &\tauSCnew{2}{5,1,3}     
& 2\\
 & & (2)^4 & 4 & 4 & 1,\, i & \tauSCnew{2}{5,4,4} & 8\\
\hline 
\multirow{2}{*}{ $\Q_2(\sqrt{-1})$ } & \multirow{2}{*}{ $-4$ } & \fp^3 & 3 & 4 & i & \tauSCnew{2}{-4,3,4} & 5\\
 & & \fp^6 & 6 & 4 & i,\,i &\tauSCnew{2}{-4,6,4} & 8 \\ 
\hline
\Q_2(\sqrt{-5}) & -20 & \fp^3 & 3 & 4 & i & \tauSCnew{2}{-20,3,4} & 5\\
\end{tabular}
\caption{\rule{0pt}{2.5ex}Nonexceptional, nonspecial inertial types over $\Q_2$}
\label{table:type-Q_2}
\end{table}

The main result of this section is then as follows.

\begin{theorem} \label{thm:nonexceptell2}
Let $E$ be an elliptic curve over $\Q_2$ with additive, potentially good reduction, conductor $\NE$, semistability defect $e_E$, and inertial type $\tau_E$.  Suppose that $e_E \ne 2,24$.  
Then $\tau_E$ is given by one of the cases in Table~\ref{tab:q2224}.
\begin{table}
\normalfont\small\renewcommand{\arraystretch}{1.25}
\begin{tabular}{>{$}c<{$}  >{$}c<{$} | >{$}c<{$} }
  \NE & e_E & \multicolumn{1}{c}{$\tau_E$} \\\hline\hline
 2^2 & 3 & \tauSCnew{2}{5,1,3}\\\hline
 2^4 & 6 & \tauSCnew{2}{5,1,3} \otimes \epschar{-4}\\\hline
2^5 & 8 & \tauSCnew{2}{-4,3,4}, \; \tauSCnew{2}{-20,3,4}\\\hline
\multirow{2}{*}{$2^6$} & 6 & \tauSCnew{2}{5,1,3} \otimes \epschar{8},\; \tauSCnew{2}{5,1,3} \otimes \epschar{-8}\\
 & 8 & \tauSCnew{2}{-4,3,4}  \otimes \epschar{8}, \; \tauSCnew{2}{-20,3,4} \otimes \epschar{8}\\\hline
\multirow{2}{*}{$2^8$} & 4 & \tauPSnew{2}{1,4,4},\; \tauPSnew{2}{1,4,4} \otimes \epschar{-4}, \; \tauSCnew{2}{5,4,4}, \; \tauSCnew{2}{5,4,4} \otimes \epschar{-4}\\
 & 8 & \tauSCnew{2}{-4,6,4},\; \tauSCnew{2}{-4,6,4} \otimes \epschar{8}\\
\end{tabular} 
\caption{\rule{0pt}{2.5ex}Inertial types over $\Q_2$ with semistability defect $e\neq 2,24$}
\label{tab:q2224}
\end{table}

In particular, $\tau_E$ is nonexceptional supercuspidal in all cases except when 
\begin{center}
$e_E=4$ and $\NE=2^8$ and $\tau_E \simeq\tauPSnew{2}{1,4,4},\tauPSnew{2}{1,4,4} \otimes \epschar{-4}$
\end{center}
in which case $\tau_E$ is principal series.  
\end{theorem}

The proof of Theorem \ref{thm:nonexceptell2} will be given in section \ref{sec:proofell25} after treating various cases in the next few sections.

\subsection{Principal series}\label{S:pseries}

Throughout the remaining sections, let $E$ be an elliptic curve over $\Q_2$ with additive, potentially good reduction, semistability defect $e_E \neq 2,24$, conductor $\NE$, and inertial type~$\tau_E$.  

We begin with the relatively easy case of principal series.
 
\begin{proposition} 
If $\tau_E$ is principal series, then $\NE = 2^8$, $e=4$, and $\tau_E \simeq\tauPSnew{2}{1,4,4}$ or $\tau_E \simeq \tauPSnew{2}{1,4,4} \otimes \epschar{-4}$. 
\label{P:pseries}
\end{proposition}

\begin{proof} 
By Lemma~\ref{L:condBound} and \eqref{E:condPS}, we have $\condexp(\tau_E) = 2k$ with $1 \leq k \leq 4$.
Thus $\tau_E = \chi|_{I_2} \oplus \chi^{-1}|_{I_2}$, where $\chi|_{I_2}$ factors through $(\Z_2/2^k\Z_2)^\times$. From Lemma~\ref{lem:grp-structures-pr-2} and Table~\ref{table:U-quo-unit-grps-pr-2} we see that 
if $k \leq 3$, then $\chi|_{I_2}$ is at most quadratic, so $e \leq 2$, a contradiction.  Therefore, $k=4$ and~$\chi|_{I_2}$ factors through
$$(\Z_2/(2^4))^\times = \langle -1 \rangle \times \langle 5 \rangle \simeq \Z/2 \times \Z/4.$$ 
The primitivity of $\chi$ forces $\chi(5) = \pm i$. 
Twisting by~$\epschar{-4}$ we can assume~$\chi(-1)=1$. Thus $\chi|_{I_2} = \chi_{(1, 4,4)}$ or $\chi_{(1, 4,4)}^{-1}$.
We conclude that $\tau \simeq\tauPSnew{2}{1,4,4}$ or $\tau \simeq \tauPSnew{2}{1,4,4} \otimes \epschar{-4}$.  
\end{proof}

\subsection{Quadratic inductions, conductor 8} \label{S:cond3}

Having dealt with the reducible case, we consider in the remaining subsections inertial types induced from a quadratic extension $\Mquad \supset \Q_2$.  In this section, we rule out the possibility that $K$ has conductor exponent $3$.

\begin{proposition} \label{P:cond3}
Suppose that $\tau_E$ is irreducible and induced from a quadratic extension~$\Mquad \supset \Q_2$.  Then $\Mquad$ is either unmarried or it has conductor exponent $2$.
\end{proposition}

\begin{proof}
Assume for purposes of contradiction that $\tau_E$ is induced from a quadratic extension $K \supset \Q_2$ of conductor exponent $3$, i.e., $d=\pm 8, \pm 40$, equivalently, ${\Mquad}=\Q_2(\sqrt{m})$ with $m = \pm 2, \pm 10$.  By Proposition~\ref{prop:non-prim-inert-type}(iv), we have $\tau \simeq \Ind_{I_{\Mquad}}^{I_{2}}(\chi|_{I_K})$ where $\chi|_{I_K} \neq \chi^s|_{I_K}$.  By Lemma~\ref{L:phi}, the order of~$\tau_E$ is $e_E=8$ as $\tau$ is irreducible and $e \neq 24$; in particular, $\Phi \simeq Q_8$ is quaternion of order $8$.

Let $f \colonequals \condexp(\chi)$.  The conductor exponent formula $\condexp(\tau) = 3 + \condexp(\chi)$ and Lemma~\ref{L:condBound} imply $k \leq 5$. 
From Lemma~\ref{lem:grp-structures-pr-2} we have that $\chi|_{I_K}$ factors through
\begin{equation} \label{eqn:ulist}
(\calO_{\Mquad}/\fp^f)^\times/U_{\frakf}  =
\begin{cases}
 \langle u \rangle \simeq \{1\}, & \text{if}\,\, f = 1; \\  
 \langle u \rangle \simeq \Z/2, & \text{if}\,\, f = 2, 3;\\  
 \langle u \rangle \simeq \Z/4, & \text{if}\,\, f = 4;\\  
\langle u \rangle \times \langle -3 \rangle \simeq \Z/4 \times \Z/2, & \text{if}\,\, f = 5;
\end{cases} 
\end{equation}
where $u = \sqrt{m} - 1$. 

We claim that $\chi^s/\chi$ is quadratic.  By the above, we have $\chi^s/\chi$ nontrivial.  From \eqref{eqn:ulist} and Lemma~\ref{lem:order4max} it follows that~$\chi^s/\chi$ is at most quadratic on inertia.  On the other hand, for the uniformizer $\sqrt{m}\in {\Mquad}$, we have 
$$(\chi^s)\spArt(\sqrt{m}) = \chi\spArt(-\sqrt{m}) = \epschar{d}\spArt(-1)\chi\spArt(\sqrt{m})$$
whence
\[ (\chi^s/\chi)\spArt(\sqrt{m}) = \epschar{d}\spArt(-1) \]
and so $\chi^s/\chi$ is at most quadratic---hence quadratic.

By Corollary~\ref{cor:factorvianorm}(c), we conclude that $\chi^s/\chi$ factors through the norm map.  By Proposition~\ref{prop:triply-imprimitive}, we conclude that $\rho_E$ is triply imprimitive, so  the projective image of $\rho_E$ is $D_2 = C_2 \times C_2$.  On the other hand, by Dokchitser--Dokchitser \cite[Lemma~1]{dd08}, there is a twist of $\rho$ which factors through $\Q_2(E[3])$, 
so $\mathrm{P}\rho \simeq \mathrm{P}\rhobar_{E,3}$.  But then $\rhobar_{E,3}(W_2) \leq \GL_2(\F_3)$ is the 2-Sylow subgroup \cite[Table~1]{dd08} and $\mathrm{P}\rhobar_{E,3}(W_2) \not \simeq C_2 \times C_2$, giving a contradiction. 
\end{proof}

\subsection{Quadratic unramified inductions} \label{S:cond0}

We now consider the case of inertial types induced from the unramified quadratic extension $\Mquad=\Q(\sqrt{5})$.  

\begin{proposition} 
Suppose $\tau_E$ is nonexceptional supercuspidal, obtained by inducing a character $\chi$ of $W_{\Mquad}$ where ${\Mquad}=\Q_2(\sqrt{5})$.  Then $\tau_E$ is given by one of the following cases:
\begin{enumalph}
 \item If $\NE = 2^2$, then $e_E=3$ and $\tau_E \simeq \tauSCnew{2}{5,1,3}$.
 \item If $\NE = 2^4$, then $e_E=6$ and $\tau_E \simeq \tauSCnew{2}{5,1,3} \otimes \epschar{-4}$.
 \item If $\NE = 2^6$, then $e_E=6$ and either $\tau_E \simeq \tauSCnew{2}{5,1,3} \otimes \epschar{8}$ or $\tau_E \simeq \tauSCnew{2}{5,1,3} \otimes \epschar{-8}$.
 \item If $\NE = 2^8$, then $e_E=4$ and either $\tau_E \simeq  \tauSCnew{2}{5,4,4}$ or $\tau_E \simeq \tauSCnew{2}{5,4,4} \otimes \epschar{-4}$.
\end{enumalph}
\label{P:cond0}
\end{proposition}

\begin{proof} 
Since ${\Mquad}$ is unramified over $\Q_2$, we are in case (iii) of Proposition~\ref{prop:non-prim-inert-type}(iii).  We conclude that $\tau$ has cyclic image of order $e_E=3,4,6$ by Lemma~\ref{L:phi} (having excluded $e_E=2$ and $e_E=24$ at the start). The conductor exponent of $\tau$ is equal to $2k$ by \eqref{E:condBC}, where $k \colonequals \condexp(\chi) \leq 4$ by Lemma~\ref{L:condBound}.  

From Corollary~\ref{cor:epsilon-d-pr-2} we have $\chi|_{\Z_2^\times} = \epschar{5}$ is trivial. Moreover, since the order of~$\chi|_{I_2}$ is also~$e > 2$ it follows from Corollary~\ref{cor:factorvianorm}(b) that all the potential candidates for~$\chi|_{I_2}$ arising below do not factor via the norm map (as required for irreducibility).  Let $\nu \colonequals (-1+\sqrt{5})/2$.  
\begin{itemize}
\item Suppose $k=1$; then $\NE = 2^2$, and the reduction is tame, hence $e_E=3$. 
From Lemma~\ref{lem:grp-structures-pr-2} and Table~\ref{table:U-quo-unit-grps-pr-2}, the character $\chi|_{I_2}$ factors through  
$$(\calO_{\Mquad}/\fp)^\times/U_{\frakf}  = \langle \nu \rangle \simeq \Z/3. $$
So $\chi(\nu) =\zeta_3^{\pm 1}$ and indeed $\chi$ does not factor through the norm.  Thus there are two conjugated possibilities for~$\chi|_{I_2}$, and we can take $\chi|_{I_2} = \chi_{(5, 1, 3)}$, which gives $\tau \simeq \tauSCnew{2}{5,1,3}$. This handles~(a).

\item Suppose $k=2$. From Lemma~\ref{lem:grp-structures-pr-2} and Table~\ref{table:U-quo-unit-grps-pr-2}, $\chi|_{I_2}$ factors through
$$(\calO_{\Mquad}/\fp^2)^\times/U_{\frakf}  = \langle \nu \rangle \simeq \Z/6.$$
From this, we see that $\chi$ is primitive if and only if $\chi(\nu) = \zeta_6^{\pm 1}$ and again, $\chi$ does not factor via the norm.  Thus there are two conjugate choices for $\chi|_{I_{\Q_2}}$, giving rise to the same type $\tau'$, showing that there is a unique type of conductor $2^4$.  But, twisting an elliptic curve with inertial type $\tauSCnew{2}{5,1,3}$ by $-1$, gives an elliptic curve of conductor $2^4$ and inertial type $\tauSCnew{2}{5,1,3} \otimes \epschar{-4}$. 
Since $\tau'$ is the unique inertial type of conductor $2^4$, we must have $\tau' = \tauSCnew{2}{5,1,3} \otimes \epschar{-4}$. This proves~(b).

\item Suppose $k=3$. From Lemma~\ref{lem:grp-structures-pr-2} and Table~\ref{table:U-quo-unit-grps-pr-2}, $\chi|_{I_2}$ factors through the quotient
\begin{equation*}
(\calO_{\Mquad}/\fp^3)^\times/ U_{\frakf} =  \langle \sqrt{5} \rangle \times \langle \nu \rangle \simeq \Z/2 \times \Z/6.
\end{equation*}
The primitivity condition here is more subtle; indeed, if
$\chi = \theta \circ \pi$ for some character~$\theta$
with $\condexp(\theta) \leq 2$ where $\pi$ is the projection defined in Remark~\ref{rem:projection_2} then,
since $(\calO_{\Mquad}/\fp^2)^\times/U_{\frakf} \simeq \Z/6$
we conclude that
$\chi(\sqrt{5} \cdot \nu^3) = \theta (\nu^6) = 1$. Thus $\chi$ is primitive if and only if $\chi(\sqrt{5}) \neq \chi(\nu)^{-3}$.
Moreover, since $\Z/4$ is not a subgroup of the above quotient, we have $e=3$ or $e=6$ and so $\chi(\nu) = \zeta_6^j$ with $1 \leq j \leq 5$ and $3 \neq j$. Thereofre, if $j = 1,5$ then $\chi(\sqrt{5}) = 1$ from the primitivity condition and if $j = 2,4$ then $\chi(\sqrt{5}) = -1$.
This gives four possibilities for $\chi|_{I_2}$ yielding two pairs of conjugate characters, and hence two possible inertial types of conductor $2^6$. But, twisting an elliptic curve with
inertial type $\tauSCnew{2}{5,1,3}$ by $2$ and $-2$ gives an elliptic curve of conductor $2^6$ 
and inertial type $\tauSCnew{2}{5,1,3} \otimes \epschar{8}$ and $\tauSCnew{2}{5,1,3} \otimes \epschar{-8}$, respectively.  So, $\tauSCnew{2}{5,1,3} \otimes \epschar{8}$ and $\tauSCnew{2}{5,1,3} \otimes \epschar{-8}$ must be the two inertial types of conductor~$2^6$, completing the proof of~(c).

\item Finally, suppose $k=4$. From Lemma~\ref{lem:grp-structures-pr-2} and Table~\ref{table:U-quo-unit-grps-pr-2}, $\chi|_{I_2}$ factors through the quotient
\begin{equation*}
(\calO_{\Mquad}/\fp^4)^\times/U_\frakf = \langle \sqrt{5} \rangle \times \langle \nu \rangle \simeq  \Z/2 \times \Z/12.
\end{equation*}
Since $\chi$ is primitive, $\chi(\nu)$ has order $4$ or $12$. In the latter case, the image of~$\tau$ has size $e=12$, a contradiction. 
So $\chi(\nu) = \pm i$ and $\chi$ does not factor via the norm.  Since there are no further constraints we can have $\chi(\sqrt{5}) = \pm 1$. This gives four possible characters. 
Similarly, one can show that $\delta = \chi \cdot  \epschar{-4}|_K$ has conductor $\fp^4$, 
satisfies $\delta|_{\Z_2^\times} = 1$ and that it does not factor through the norm. Hence the possibilities for $\chi|_{I_2}$ are 
\[
\chi_{(5, 4, 4)},\,\, \chi_{(5, 4, 4)}^s,\,\, \chi_{(5, 4, 4)} \cdot \epschar{-4},\,\, \text{or}\,\, (\chi_{(5, 4, 4)} \cdot \epschar{-4})^s,
\]
therefore $\tau_E \simeq  \tauSCnew{2}{5,4,4}$ or $\tau_E \simeq \tauSCnew{2}{5,4,4} \otimes \epschar{-4}$, as desired. 
(Note that the twisted types $\tauSCnew{2}{5,4,4} \otimes \varepsilon_{\pm 8}$ also have conductor $2^8$ 
but they do not appear above due to the relations
$\chi_{(5, 4, 4)}^s = \chi_{(5, 4, 4)} \epschar{-8}$ and $(\chi_{(5, 4, 4)} \epschar{-4})^s = \chi_{(5, 4, 4)} \epschar{8}$.)
\end{itemize}
This exhausts the possible cases and completes the proof.
\end{proof}

\subsection{Quadratic inductions, conductor 4} \label{S:cond2}

We conclude with the case of conductor~$4=2^2$.

\begin{proposition}
Suppose $\tau_E$ is nonexceptional supercuspidal type, obtained by inducing a character $\chi$ of $W_{\Mquad}$ 
where $K \supseteq \Q_2$ has conductor exponent $2$.  Then $e_E=8$ and $\tau_E$ is given by one of the following cases:
\begin{enumalph}
 \item If $\NE = 2^5$, then $\tau_E \simeq \tauSCnew{2}{-4,3,4}$ or $\tau_E \simeq \tauSCnew{2}{-20,3,4}$.
 \item If $\NE = 2^6$, then $\tau_E \simeq \tauSCnew{2}{-4,3,4}  \otimes \epschar{8}$ or $\tau_E \simeq \tauSCnew{2}{-20,3,4} \otimes \epschar{8}$.
 \item If $\NE = 2^8$, then $\tau_E \simeq \tauSCnew{2}{-4,6,4}$ or $\tau_E \simeq \tauSCnew{2}{-4,6,4} \otimes \epschar{8}$.
\end{enumalph}
\label{P:cond2}
\end{proposition}

\begin{proof} 
By assumption $e \neq 24$, hence by Lemma~\ref{L:phi} and Proposition~\ref{prop:non-prim-inert-type} we must have $e_E=8$.

The quadratic extensions $K \supset \Q_2$ of conductor~$2^2$ are $K=\Q_2(\sqrt{m})$ with $m=-1, -5$ corresponding to $d=-4,-20$.  In both cases, we must have $\chi|_{\Z_2^\times} = \epschar{-4}$, which implies that  $\chi(-1) = \chi(3) = -1$
by Corollary~\ref{cor:epsilon-d-pr-2}.
In particular, all the candidates for~$\chi$ we will find below do not factor via the norm map by Corollary~\ref{cor:factorvianorm}(a).
By the conductor exponent formula~\eqref{E:condBC} and Lemma~\ref{L:condBound}, we have $k \colonequals \condexp(\chi) \leq 6$. Since all characters with conductor exponent 
$\leq 2$ have $\chi(-1)=1$, we must have $3 \leq k \leq 6$. Furthermore, from Table~\ref{table:U-quo-unit-grps-pr-2} we also see there are no primitive characters for $k=5$ for both values of~$m$, hence $k = 3,4,6$.

\begin{itemize}
\item Suppose $d=-4$ and $k=3$. By Lemma~\ref{lem:grp-structures-pr-2} and Table~\ref{table:U-quo-unit-grps-pr-2}, $\chi|_{I_K}$ factors through
$$(\calO_{\Mquad}/\fp^3)^\times/U_\frakf = \langle \sqrt{m} \rangle \simeq \Z/4.$$ 
Since $\chi$ is primitive, we must have $\chi(\sqrt{m})= \pm i$. 
Therefore, there are two possible conjugate choices. Taking $\chi|_{I_K} = \chi_{(-4,3,4)}$,
we obtain $\tau \simeq \tauSCnew{2}{-4,3,4}$. This proves~(a)
for $m=-1$.

\item Suppose $d=-4$ and $k=4$. By Lemma~\ref{lem:grp-structures-pr-2} and Table~\ref{table:U-quo-unit-grps-pr-2}, $\chi|_{I_K}$ factors through
\[ (\calO_{\Mquad}/\fp^4)^\times/U_\frakf  = \langle \sqrt{m} \rangle \times \langle 2\sqrt{m} - 1\rangle \simeq \Z/4 \times \Z/2. \]
Since $\chi$ is primitive, we have $\chi(2\sqrt{m} - 1) = -1$. Furthermore, since $\chi(-1)=\chi(\sqrt{m})^2 = -1$, we must have $\chi(\sqrt{m})= \pm i$.
Therefore, there are two conjugate choices, and a unique type of conductor~$2^6$.
But, twisting an elliptic curve with inertial type $\tauSCnew{2}{-4,3,4}$ by $2$ gives an elliptic curve with conductor~$2^6$ and inertial type $\tauSCnew{2}{-4,3,4}  \otimes \epschar{8}$. 
By the uniqueness of the type at conductor~$2^6$, we must have $\tau \simeq \tauSCnew{2}{-4,3,4}  \otimes \epschar{8}$, proving the first type in~(b).

\item Suppose $d=-4$ and $k=6$. By Lemma~\ref{lem:grp-structures-pr-2} and Table~\ref{table:U-quo-unit-grps-pr-2}, $\chi|_{I_K}$ factors through
\[ (\calO_{\Mquad}/\fp^6)^\times/U_\frakf  = \langle \sqrt{m} \rangle \times \langle 2\sqrt{m} - 1\rangle \simeq \Z/4 \times \Z/4. \]
Since $\chi$ is primitive, we must have $\chi(2\sqrt{m} - 1) = \pm i$. Analogously to the case $k = 4$, we also have $\chi(\sqrt{m})= \pm i$.
This gives rise to four characters. 
Note that $\delta = \chi \cdot  \epschar{8}|_K$ also has conductor~$\fp^6$,
and satisfies $\delta|_{\Z_2^\times} = \epschar{-4}$ and does not factor through the norm. We conclude that the four possibilities  for $\chi|_{I_K}$ are
\[
 \chi_{(-4,6, 4)}, \quad \chi_{(-4,6, 4)}^s, \quad \chi_{(-4,6, 4)} \cdot \epschar{8}|_{K}, \quad (\chi_{(-4,6, 4)} \cdot \epschar{8}|_{K})^s,
\]
yielding $\tau = \tauSCnew{2}{-4,6,4}$ or $\tauSCnew{2}{-4,6,4} \otimes \epschar{8}$. This proves (c) for $d=-4$.  

\item Suppose $d=-20$ and $k = 3, 4$.  Then the same argument as for $m=-1$ applies, using again Lemma~\ref{lem:grp-structures-pr-2} and Table~\ref{table:U-quo-unit-grps-pr-2}
for the group structures. This completes the proof of (a) and (b) for $d=-20$.

\item Finally, suppose $d=-20$ and $k=6$. By Lemma~\ref{lem:grp-structures-pr-2} and Table~\ref{table:U-quo-unit-grps-pr-2}, $\chi|_{I_K}$ factors through
\[ (\calO_{\Mquad}/\fp^6)^\times/U_\frakf  = \langle \sqrt{m} \rangle \times \langle 2\sqrt{m} - 1\rangle \simeq \Z/4 \times \Z/4. \]
Since~$\chi$ is primitive we have $\chi(2\sqrt{m} - 1) = \pm i$ as for $m=-1$.  We also have
\[
 \chi(\sqrt{m})^2 = \chi(-5) = \chi(3) = -1 \implies \chi(\sqrt{m}) = \pm i.
\]

\noindent
We claim that $\chi^s/\chi$ is quadratic. Indeed, it follows from Lemma~\ref{lem:order4max} and the fact that~$\tau$ is irreducible that~$\chi^s/\chi$ is quadratic on inertia.  We now compute $(\chi^s/\chi)(\pi)$,
where $\pi = 1 - \sqrt{m}$ is a uniformizer:
\begin{align*}
(\chi^s/\chi)(\pi) &= \chi(s(\pi)/\pi) = \chi((1 + \sqrt{m})/(1 - \sqrt{m})) = \chi(u(1 - \sqrt{m})/(1 - \sqrt{m}))\\
&= \chi(u),
\end{align*}
where $u = -(2 - \sqrt{m})/3 \in \calO_K^\times$. As elements of $\calO_{\Mquad}/\fp^6)^\times/U_\frakf$, we have $u= \sqrt{m}^3 \cdot (2\sqrt{m}-1)$ hence $$\chi(u)= \chi(\sqrt{m})^3 \cdot \chi(2\sqrt{m}-1) = (\pm i)^3(\pm i) = \pm 1,$$ showing that $\chi^s/\chi$ is quadratic, as claimed.  

Since $\chi^s/\chi$ is quadratic, it factors via the norm map by Corollary~\ref{cor:factorvianorm}. So $\rho_E$
is triply imprimitive by Proposition~\ref{prop:triply-imprimitive}.  This leads to a contradiction as in the proof of Proposition~\ref{P:cond3}.  We conclude there are no types arising from an elliptic curve for $k=6$ and $d=-20$.
\end{itemize}
This completes the proof by cases.
\end{proof}

\begin{remark} In Proposition~\ref{P:cond2}, we always have $e = 8$.  To ensure that we do not get the same kind of contradiction that arises in the last case of the proof of this proposition, to show that this type occurs we must check the order of~$\chi^s/\chi$ in all cases.  Instead, we will show in section~\ref{sec:non-exp:2} that there is an elliptic curve with that type.
\end{remark}

\subsection{Proof of theorem} \label{sec:proofell25}

We are now ready to prove the main result of section~\ref{sec:inert2nonexcept}. 

\begin{proof}[Proof of Theorem~\textup{\ref{thm:nonexceptell2}}]

Since $e \neq 24$, by Lemma~\ref{L:exceptional=24} and Proposition~\ref{prop:non-prim-inert-type},  $\tau_E$ is either principal series or nonexceptional supercuspidal.  

The case when $\tau_E$ is principal series is treated by Proposition~\ref{P:pseries}. 

If $\tau_E$ is nonexceptional supercuspidal with conductor $2^8$ induced from a quadratic extension~$K$, then by Proposition~\ref{P:cond3}, $K$ has conductor $1$ or $4$: these are covered in Propositions~\ref{P:cond0}(d) and~\ref{P:cond2}(c), respectively.  

We are left with the nonexceptional supercuspidal types of conductor~$\neq 2^8$, which are covered by Propositions~\ref{P:cond0}(a)--(c) and~\ref{P:cond2}(a)--(b).  
\end{proof}

\subsection{Explicit realization} \label{sec:non-exp:2}

In Table~\ref{table:types-Q2:sc} we match each nonexceptional inertial type with a representative elliptic curve over $\Q_2$, with code online \cite{ourcode}.

We recall the notion of descent from section~\ref{sec:exm3}.

\begin{table}
\normalfont\small\renewcommand{\arraystretch}{1.25}
\begin{tabular}{ >{$}c<{$} || >{$}c<{$} | >{$}c<{$} | >{$}c<{$} | >{$}c<{$} | >{$}c<{$} || >{$}c<{$} }
\tau & e & \condexp(\tau) & \Nm_{L|K}(\calO_L^\times)/U_\frakf & L' & \Gal(L\,|\,\Q_2) & E \\
\hline\hline
\text{trivial} &  1 & 0 & - & \LMFDBL{2.1.0.1} & \LMFDBG{1T1} \simeq C_1 & \LMFDBE{11a1} \\
\hline
\epschar{-1} &  2 & 4 & - & \LMFDBL{2.2.2.1} & \LMFDBG{2T1} \simeq C_2 & \LMFDBE{176b2}\\
\hline
\epschar{2} &  2  & 6 & - & \LMFDBL{2.2.3.1} & \LMFDBG{2T1} \simeq C_2 & \LMFDBE{704a2}\\
\hline
\epschar{-2} &  2 & 6  & - & \LMFDBL{2.2.3.3} & \LMFDBG{2T1} \simeq C_2 & \LMFDBE{704k2}\\
\hline
\tauSCnew{2}{5,1,3} &  3  & 2  & \textup{trivial} & \LMFDBL{2.3.2.1} & \LMFDBG{3T2} \simeq S_3 & \LMFDBE{20a1}\\
\hline
\tauSCnew{2}{5,4,4} &  4  & 8 &\langle (1, 4)\rangle \simeq \Z/6  & \LMFDBL{2.4.11.18} & \LMFDBG{4T3} \simeq D_4 & \LMFDBE{256a1} \\
\hline
\tauSCnew{2}{5,4,4} \otimes \epschar{-4}  &  4  & 8 & \langle (1, 10)\rangle \simeq \Z/6  & \LMFDBL{2.4.11.17} & \LMFDBG{4T3} \simeq D_4 & \LMFDBE{256d1} \\
\hline
\tauPSnew{2}{1,4,4}  &  4 & 8 & \langle (1, 0)\rangle \simeq \Z/2 & \LMFDBL{2.4.11.1} & \LMFDBG{4T1} \simeq C_4 & \LMFDBE{768b1}\\
\hline
\tauPSnew{2}{1,4,4} \otimes \epschar{-4} &  4 & 8 & \langle (1, 2)\rangle \simeq \Z/2 & \LMFDBL{2.4.11.2} & \LMFDBG{4T1} \simeq C_4  & \LMFDBE{768h1} \\
\hline
\tauSCnew{2}{5,1,3} \otimes \epschar{-4}  &  6 & 4 & \textup{trivial} & \LMFDBL{2.6.8.1} & \LMFDBG{6T3} \simeq D_6 & \LMFDBE{80b1}\\
\hline
\tauSCnew{2}{5,1,3} \otimes \epschar{8}  &  6  & 6 & \langle (0, 3)\rangle \simeq \Z/2 & \LMFDBL{2.6.11.1} & \LMFDBG{6T3} \simeq D_6 & \LMFDBE{320c2}\\
\hline
\tauSCnew{2}{5,1,3} \otimes \epschar{-8} &  6 & 6 & \langle (1, 0) \rangle \simeq \Z/2 & \LMFDBL{2.6.11.9} & \LMFDBG{6T3} \simeq D_6 & \LMFDBE{320f2} \\
\hline
\tauSCnew{2}{-20, 3, 4} &  8 & 5 & \textup{trivial} & \LMFDBL{2.8.16.65} & \LMFDBG{8T8} \simeq 2\cdot D_4& \LMFDBE{96a1} \\
\hline
\tauSCnew{2}{-4, 3, 4}  &  8 & 5 & \textup{trivial} & \LMFDBL{2.8.16.66} & \LMFDBG{8T8} \simeq 2 \cdot D_4& \LMFDBE{288a1} \\
\hline
\tauSCnew{2}{-20, 3, 4} \otimes \epschar{8}  &  8 & 6 & \langle (2, 1)\rangle \simeq \Z/2 & \LMFDBL{2.8.18.74} & \LMFDBG{8T8} \simeq 2 \cdot D_4 & \LMFDBE{192a2} \\
\hline
\tauSCnew{2}{-4, 3, 4} \otimes \epschar{8} &  8  & 6 & \langle (2, 1)\rangle \simeq \Z/2 & \LMFDBL{2.8.18.73} & \LMFDBG{8T8} \simeq 2 \cdot D_4& \LMFDBE{576f2}  \\
\hline
\tauSCnew{2}{-4,6,4} \otimes \epschar{8}  &  8  & 8 & \langle (3, 3)\rangle \simeq \Z/4& \LMFDBL{2.8.24.66} & \LMFDBG{8T8} \simeq 2 \cdot D_4& \LMFDBE{256b2} \\
\hline
\tauSCnew{2}{-4,6,4} &  8 & 8 & \langle (3, 1)\rangle \simeq \Z/4 & \LMFDBL{2.8.24.68} & \LMFDBG{8T8} \simeq 2 \cdot D_4& \LMFDBE{256c1} \\
\end{tabular}
\caption{\rule{0pt}{2.5ex}Types, defining fields, and elliptic curves realizing each nonexceptional inertial type over $\Q_2$ in the case of potentially good reduction}
\label{table:types-Q2:sc}
\end{table}

\begin{proposition} \label{prop:2eachfield}
For $p=2$ and each nonexceptional inertial type $\tau$ arising from an elliptic curve
with additive, potentially good reduction, there is a unique descent $L'$ of the inertial field.  Moreover, either $L'=L$ is Galois or the compositum $L=\Q_4 L'$ is Galois, where $\Q_4$ is the quadratic unramified extension of $\Q_2$.
\end{proposition}

\begin{proof}
The proof is similar to Proposition~\ref{prop:3eachfield}.
\end{proof}

With the list of fields in hand, the proof of Table~\ref{table:types-Q2:sc} is similar to section~\ref{sec:exm3}, and we now summarize it. We note however that, in contrast with section~\ref{sec:exm3}, in the argument below we do not need the description of the norm groups of all the rows in the table,
as many rows are deduced by taking quadratic twists of other rows (for which we do use that information); we decided to include the norm groups in all cases for completion sake.

The curve \LMFDBE{11a1} has good reduction at $2$ proving the first row and
twisting by $-1$, $2$, and~$-2$ yields the next three rows.  To complete the correspondence, using norm group computations we will find which type corresponds to each field.

We start with the principal series case. The curve \LMFDBE{768b1} over~$\Q_2$ has conductor $2^8$ and obtain good reduction over a cyclic extensions $L_\tau \supseteq \Q_2$ of degree 4 hence its inertial type~$\tau$ is a principal series of conductor $2^8$  determined by a character $\chi|_{I_2}$ of order~4. Moreover, setting $K=\Q_2$ and $\frakf = 2^4$, by local class field theory, we have
\begin{align*}
\ker (\chi^A|_{(\calO_K/\frakf)^\times}) = \Nm_{L_\tau|K}(\calO_{L_\tau}^\times)  \hookrightarrow (\calO_K/\frakf)^\times,
\end{align*}
where we used that $U_\frakf = 1$ in this case. Comparing the norm group obtained from $L_\tau$ to the character defined in the first row of Table~\ref{table:type-Q_2} we conclude that $\tau = \tauPSnew{2}{1,4,4}$. Twisting \textsf{768b1} by $-1$ shoes that the inertial type of \LMFDBE{768h1} is $\tauPSnew{2}{1,4,4} \otimes \epschar{-4}$.

We are now left with the supercuspidal types. Following an argument and calculations analogous to section~\ref{sec:exm3} we compute norm groups of the form  $\Nm_{L|K}(\calO_L^\times)/U_\frakf$ inside $(\calO_K/\frakf)^\times/U_\frakf$ where the latter is given by the group structure and generators in Table~\ref{table:U-quo-unit-grps-pr-2}; these norm groups uniquely identify the field $L$ as an extension of~$K$. Comparing the output of this calculation to the definition of the characters in Table~\ref{table:type-Q_2} plus taking adequate quadratic twists establishes all rows except those corresponding to \LMFDBE{96a1}, \LMFDBE{288a1}, \LMFDBE{192a2}, and \LMFDBE{576f2}; since the latter two curves are quadratic twits by~$\epschar{8}$ of the first two, it suffices to determine the types for \textsf{96a1} and \textsf{288a1}. There is ambiguity here because the fields of good reduction of these curves both contain $\Q_2(\sqrt{d})$ for $d=-4,-20$ and the two resulting norm groups are~$\{0\}$.  Finally, an argument analogous to that at the end of section~\ref{sec:exm3} for the types $\tauSCnew{3}{\pm 3,2,6}$
completes the proof---here we use that $\Gal(L\,|\,\Q_2) \simeq  2 \cdot D_4$ modulo its center has order 8.

The followiin corollary is analogous to Corollary~\ref{cor:algQ3} and follows from the previous discussion.

\begin{corollary} \label{cor:algQ2}
Let $E$ be an elliptic curve over $\Q_2$ with potentially good reduction. Assume that $E$ semistability defect $e_E \neq 2$.
Then there is a unique field $L'$ in Table~\ref{table:types-Q2:sc} of minimal degree such that $E$ obtains good reduction over $L'$, and $\tau_E$ is given by the type that corresponds to this $L'$.
\end{corollary}

\section{Exceptional inertial types for \texorpdfstring{$E$ over $\Q_2$}{EoverQQ2}}
\label{S:exceptionalThm}

Finally, we consider exceptional inertial types which arise only for $p=2$.

\subsection{Setup and result}

Let $r=\pm 1, \pm 2$ and define the following elliptic curves over $\Q_2$
\begin{equation} \label{E:exceptCurves}
 E_{1,r} \colon ry^2 = x^3 - 3x - 1 \qquad \text{and} \qquad  E_{2,r} \colon ry^2 = x^3 + 3x + 2.
\end{equation}
These curves have potentially good reduction with semistability defect $e=24$. We denote by $\tau_{i,r}$ the inertial type of~$E_{i,r}$. We have $N=0$ for all~$i,r$ as above. For reasons that will shortly be clear, we also abbreviate $\tauex{1}:=\tau_{1,1}$ and $\tauex{2} := \tau_{2,1}$. Our final result is as follows.

\begin{theorem} \label{thm:exceptell2}
Let $E$ be an elliptic curve over $\Q_2$ with potentially good reduction, semistability defect $e_E = 24$, conductor~$\NE$ and inertial type~$\tau_E$.  Then one of the following holds.
\begin{enumalph}
 \item If $\NE = 2^3$, then $\tau_E \simeq \tauex{1}$.
 \item If $\NE = 2^4$, then $\tau_E \simeq \tauex{1} \otimes \epschar{-4}$.
 \item If $\NE = 2^6$, then $\tau_E \simeq \tauex{1} \otimes \epschar{8}$ or $\tau_E = \tauex{1} \otimes \epschar{-8}$.
 \item If $\NE = 2^7$, then $\tau_E$ is isomorphic to one of $\tauex{2}$, $\tauex{2} \otimes \epschar{-4}$, $\tauex{2} \otimes \epschar{8}$ or~$\tauex{2} \otimes \epschar{-8}$.
\end{enumalph}
\label{T:primtypes}
\end{theorem}
 
\begin{proof} For $i = 1, 2$ and $r = \pm 1, \pm 2$, let $K_{i, r} := \Q_2(E_{i, r})$, where $E_{i,r}$ is given by \eqref{E:exceptCurves}. 
The fields $K_{i,r}$ give the set of all $\widetilde{S_4}$-extensions of $\Q_2$, where $\widetilde{S_4} \simeq \GL_2(\F_3)$ is the double cover of $S_4$: see Bayer--Rio \cite[Table 10]{BayRio}).
Table~\ref{table:types-Q2:ex} lists a polynomial whose splitting field is $K_{i,r}$ for each $i$ and $r$. 

Let $K=\Q_2(E[3])$, $G=\Gal(K\,|\,\Q_2)$
and $L = \Q_2^{\textup{un}} K$ be the inertial field of $E$.
From the proof of Lemma~\ref{L:exceptional=24}, we know that $G \simeq \widetilde{S_4}$ 
a double cover of $\mathrm{P}(\rhobar_{E,3}) \simeq S_4$. Therefore, there is a choice of~$i,r$ such that  both $\tau$ and~$\tau_{i,r}$ 
fix the extension $L \supset \Qun_2$. We have $\Gal(L\,|\,\Qun_2) \simeq \Phi \simeq \SL_2(\F_3)$ by Lemma~\ref{L:phi}, and since there is only 
one irreducible $\GL_2(\C)$-representation of $\SL_2(\F_3)$ whose image is contained in $\SL_2(\C)$,
we conclude that $\tau \simeq \tau_{i,r}$.

Note that, for $r=-1,2,-2$, the curve $E_{1,r}$ is the quadratic twist of $E_{1,1}$ by $-4,8,-8$, respectively, therefore
$\tau_{1,-1} \simeq \tauex{1} \otimes \epschar{-4}$, 
$\tau_{1,2} \simeq \tauex{1} \otimes \epschar{8}$ and
$\tau_{1,-2} \simeq \tauex{1} \otimes \epschar{-8}$.
Similarly, we obtain 
$\tau_{2,-1} \simeq \tauex{2} \otimes \epschar{-4}$, 
$\tau_{2,2} \simeq \tauex{2} \otimes \epschar{8}$ and
$\tau_{2,-2} \simeq \tauex{2} \otimes \epschar{-8}$.

Finally, observe that
the conductor of  $E_{1,1}$ is $2^3$ and that of $E_{2,1}$ is $2^7$, 
thus the eight types split in the 4 cases of the theorem according to their conductors. 
\end{proof}

\subsection{Explicit characters} \label{sec:explchar2}

Recall that, as in previous sections, we aim for an explicit description of the types in Theorem~\ref{T:primtypes} in terms of characters. 
As explained in section~\ref{S:exceptional}, an exceptional type is determined by a triple $(L,M,\chi)$, where $L\,|\,\Q_2$ is a cubic extension, 
$M\,|\,L$ a quadratic extension and $\chi \colon W_M \rightarrow \C^\times$ a character such that $\chi \neq \chi^s$ where $s$ is conjugation on $M\,|\,L$.
Furthermore, we only need to specify~$\chi$ on~$I_M$. 

\begin{lemma}\label{lem:subfields}We keep the above notations. All fields $K_{i, r}$ contain the cubic field $F_2$, and the unique
unramified quadratic extension $\Q_4(\sqrt[3]{2})$ of $F_2$. Moreover: 
\begin{enumalph}
\item The fields $K_{1, r}$, $r = \pm 1, \pm 2$, contain the two ramified quadratic extensions $\Q_2(\sqrt[3]{2},\alpha_{\pm})$ of $F_2$, where $\alpha_{-}$ is a root of
$x^2 + \sqrt[3]{2}x + \sqrt[3]{2}$ and $\alpha_{+}$ a root of $x^2 + \sqrt[3]{2}x + \sqrt[3]{2}^2 + \sqrt[3]{2}$.
\item The fields $K_{2, r}$, $r = \pm 1,\pm 2$, contain the two ramified quadratic extensions $\Q_2(\sqrt[3]{2},\beta_{\pm})$ of $F_2$, where $\beta_{-}$ is a root of
$x^2 + 2x + \sqrt[3]{2} + 2$ and $\beta_{+}$ a root of $x^2 + 2x + \sqrt[3]{2} + 6$.
\end{enumalph}
\end{lemma}

\begin{proof}
This follows from direct calculations. 
\end{proof}

\begin{lemma}\label{lem:grp-structures-pr-2:ex} Let $K \supseteq F_2$ be one of the quadratic extensions given in Lemma~\textup{\ref{lem:subfields}}.  
Then, Tables \textup{\ref{table:norm-grps-pr-2:ex}}, \textup{\ref{table:cubic-quo-unit-grps-pr-2:ex}}, \textup{\ref{table:cubic-gen-unit-grps-pr-2:ex}} 
and \textup{\ref{table:U-quo-unit-grps-pr-2:ex}} give the structure and generators for the groups $U_{\frakf}$, $(\calO_K/\frakf)^\times/U_\frakf$
and $(\Z_2[\sqrt[3]{2}]/\frakq)^\times/U_\frakf$, respectively.
\end{lemma}

\begin{proof}
The proof uses the same kind of induction as in Lemma~\ref{lem:grp-structures-3}.
\end{proof}

\begin{table}
\normalfont\small\renewcommand{\arraystretch}{1.25}
\begin{tabular}{ >{$}c<{$} || >{$}c<{$} | >{$}c<{$} || >{$}r<{$} >{$}l<{$} }
K\,|\,F_2 & \mathfrak{f} & f & \multicolumn{2}{c}{$U_{\frakf}$} \\\hline\hline
\multirow{10}{*}{$\Q_4(\sqrt[3]{2})$} & \multirow{10}{*}{$\frakp^f$} & 1 & \multicolumn{2}{c}{trivial} \\
&   & 2 & \langle u_4 \rangle \!\!\!\!\!&\simeq  \Z/2 \\
&   & 3 & \langle u_4 \rangle \!\!\!\!\!&\simeq \Z/4 \\
&   & 4 & \langle u_3\rangle \times \langle u_4 \rangle \!\!\!\!\!&\simeq \Z/2 \times \Z/4 \\
&   & 5 & \langle u_3\rangle \times \langle u_4 \rangle \!\!\!\!\!&\simeq \Z/2 \times \Z/8 \\
&   &  6 & \langle u_2\rangle \times \langle _3 \rangle \times \langle u_4 \rangle \!\!\!\!\!&\simeq \Z/2 \times \Z/2 \times \Z/8 \\
&   &  7 &  \langle u_1 \rangle \times \langle u_2 \rangle\times \langle u_3 \rangle \times \langle u_4 \rangle \!\!\!\!\!&\simeq \Z/2 \times \Z/2 \times \Z/2 \times \Z/8 \\
&   &  8 &  \langle u_1 \rangle \times \langle u_2 \rangle\times \langle u_3 \rangle \times \langle u_4 \rangle \!\!\!\!\!&\simeq \Z/2 \times \Z/2 \times \Z/2 \times \Z/16 \\
&   &  9 &  \langle u_1 \rangle \times \langle u_2 \rangle\times \langle u_3 \rangle \times \langle u_4 \rangle \!\!\!\!\!&\simeq \Z/2 \times \Z/2 \times \Z/4 \times \Z/16 \\
&   &  \ge 10 &  \langle u_1 \rangle \times \langle u_2 \rangle\times \langle u_3 \rangle \times \langle u_4 \rangle \!\!\!\!\!&\simeq\\
& & & \multicolumn{2}{c}{\hfill $\Z/2 \times \Z/2^{\lfloor \frac{f-10}{3} \rfloor + 2} \times \Z/2^{\lfloor \frac{f-9}{3}\rfloor + 2} \times \Z/2^{\lfloor\frac{f -8}{3}\rfloor + 4}$}\\
\hline
\multirow{9}{*}{$\Q_2(\sqrt[3]{2}, \alpha_{\pm})$} & \multirow{9}{*}{$\frakp^f$} & 1, 2, 3, 4 & \multicolumn{2}{c}{trivial} \\
&   &  5, 6 &  \langle u_4 \rangle \!\!\!\!\!&\simeq \Z/2 \\
&   &  7, 8 & \langle u_3 \rangle \times \langle u_4 \rangle \!\!\!\!\!&\simeq \Z/2 \times \Z/2 \\
&   & 9, 10 & \langle u_3 \rangle \times \langle u_4 \rangle \!\!\!\!\!&\simeq \Z/2 \times \Z/4 \\
&   &  11, 12 & \langle u_2\rangle \times \langle u_3 \rangle \times \langle u_4 \rangle \!\!\!\!\!&\simeq \Z/2 \times \Z/2 \times \Z/4 \\
&   &  13, 14 &  \langle u_1 \rangle \times \langle u_2 \rangle\times \langle u_3 \rangle \times \langle u_4 \rangle \!\!\!\!\!&\simeq \Z/2 \times \Z/2 \times \Z/2 \times \Z/4 \\
&   &  15, 16 &  \langle u_1 \rangle \times \langle u_2 \rangle\times \langle u_3 \rangle \times \langle u_4 \rangle \!\!\!\!\!&\simeq \Z/2 \times \Z/2 \times \Z/2 \times \Z/8 \\
&   &  17, 18 &  \langle u_1 \rangle \times \langle u_2 \rangle\times \langle u_3 \rangle \times \langle u_4 \rangle \!\!\!\!\!&\simeq  \Z/2 \times \Z/2 \times \Z/4 \times \Z/8 \\
&   &  \ge 19 &  \langle u_1 \rangle \times \langle u_2 \rangle\times \langle u_3 \rangle \times \langle u_4 \rangle \!\!\!\!\!&\simeq\\
& & & \multicolumn{2}{c}{\hfill $\Z/2 \times \Z/2^{\lfloor \frac{f - 19}{6}\rfloor + 2} \times \Z/2^{\lfloor \frac{f - 17}{6}\rfloor + 2} \times \Z/2^{\lfloor\frac{f - 15}{6}\rfloor + 3}$}\\
\hline
\multirow{14}{*}{$\Q_2(\sqrt[3]{2}, \beta_{\pm})$} & \multirow{14}{*}{$\frakp^f$} & 1, 2 & \multicolumn{2}{c}{trivial} \\
&   & 3, 4 & \langle u_4 \rangle \!\!\!\!\!&\simeq  \Z/2 \\
&   & 5, 6 & \langle u_4 \rangle \!\!\!\!\!&\simeq \Z/4 \\
&   & 7, 8 & \langle u_3 \rangle \times \langle u_4\rangle \!\!\!\!\!&\simeq \Z/2 \times \Z/4 \\
&   & 9, 10, 11, 12 & \langle u_3 \rangle \times \langle u_4\rangle \!\!\!\!\!&\simeq \Z/2 \times \Z/8 \\
&   &  13, 14 & \langle u_2\rangle \times \langle u_3 \rangle \times \langle u_4 \rangle \!\!\!\!\!&\simeq \Z/2 \times \Z/2 \times \Z/8 \\
&   &  15, 16 & \langle u_2\rangle \times \langle u_3 \rangle \times \langle u_4 \rangle \!\!\!\!\!&\simeq \Z/2 \times \Z/2 \times \Z/16 \\
&   & 17, 18 &  \langle u_1 \rangle \times \langle u_2 \rangle\times \langle u_3 \rangle \times \langle u_4 \rangle \!\!\!\!\!&\simeq \Z/2 \times \Z/2 \times \Z/2 \times \Z/16 \\
&   &  19, 20 &  \langle u_1 \rangle \times \langle u_2 \rangle\times \langle u_3 \rangle \times \langle u_4 \rangle \!\!\!\!\!&\simeq \Z/2 \times \Z/2 \times \Z/4 \times \Z/16 \\
&   &  21, 22 &  \langle u_1 \rangle \times \langle u_2 \rangle\times \langle u_3 \rangle \times \langle u_4 \rangle \!\!\!\!\!&\simeq \Z/2 \times \Z/2 \times \Z/4 \times \Z/32 \\
&   &  23, 24 &  \langle u_1 \rangle \times \langle u_2 \rangle\times \langle u_3 \rangle \times \langle u_4 \rangle \!\!\!\!\!&\simeq \Z/2 \times \Z/4 \times \Z/4 \times \Z/32 \\
&   &  25, 26 &  \langle u_1 \rangle \times \langle u_2 \rangle\times \langle u_3 \rangle \times \langle u_4 \rangle \!\!\!\!\!&\simeq \Z/2 \times \Z/4 \times \Z/8 \times \Z/32 \\
&   &  27, 28 &  \langle u_1 \rangle \times \langle u_2 \rangle\times \langle u_3 \rangle \times \langle u_4 \rangle \!\!\!\!\!&\simeq \Z/2 \times \Z/4 \times \Z/8 \times \Z/64 \\
&   &  \ge 29 &  \langle u_1 \rangle \times \langle u_2 \rangle\times \langle u_3 \rangle \times \langle u_4 \rangle \!\!\!\!\!&\simeq\\ 
& & & \multicolumn{2}{c}{\hfill $\Z/2 \times \Z/2^{\lfloor \frac{f - 29}{6}\rfloor + 3}  \times \Z/2^{\lfloor \frac{f - 25}{6}\rfloor + 3} \times \Z/2^{\lfloor\frac{f - 21}{6}\rfloor + 5}$}\\

\end{tabular}
\caption{\rule{0pt}{2.5ex}Group structure of $U_{\frakf}$ for the quadratic extensions $K\,|\,F_2$ contained in $K_{1, r}$ and $K_{2,r}$, with $r = \pm1,\pm2$}
\label{table:norm-grps-pr-2:ex}
\end{table}

\begin{table}
\normalfont\small\renewcommand{\arraystretch}{1.25}
\begin{tabular}{ >{$}c<{$} || >{$}c<{$} | >{$}c<{$} || >{$}r<{$} >{$}l<{$} }
{\Mquad} & \mathfrak{f} & f & \multicolumn{2}{c}{$(\Z_2[\sqrt[3]{2}]/\frakq)^\times/U_{\frakf}$} \\\hline\hline
\Q_4(\sqrt[3]{2}) & \frakp^f & \geq 1 & \multicolumn{2}{c}{trivial} \\\hline
\multirow{2}{*}{$\Q_2(\sqrt[3]{2}, \alpha)$} &\multirow{2}{*}{$\frakp^f$} & 1,2 & \multicolumn{2}{c}{trivial} \\
&&\ge 3&\langle 1 - \sqrt[3]{2}\rangle\!\!\!\!&\simeq \Z/2\\\hline
\multirow{2}{*}{$\Q_2(\sqrt[3]{2}, \beta)$} &\multirow{2}{*}{$\frakp^f$} & 1,2,3,4, 5, 6, 7, 8, 9, 10 & \multicolumn{2}{c}{trivial} \\
&&\ge 11 & \langle 1 + \sqrt[3]{2}^2\rangle\!\!\!\!&\simeq \Z/2\\

\end{tabular}
\caption{\rule{0pt}{2.5ex}Group structure of $(\Z_2[\sqrt[3]{2}]/\frakq)^\times/U_{\frakf}$ for the quadratic extensions $K\,|\,F_2$ contained in $K_{1, r}$ and $K_{2,r}$, with $r = \pm1,\pm2$}
\label{table:cubic-quo-unit-grps-pr-2:ex}
\end{table}

\begin{table}
\normalfont\small\renewcommand{\arraystretch}{1.25}
\begin{tabular}{ >{$}c<{$} || >{$}c<{$} }
\text{Generators}& \Q_4(\sqrt[3]{2}) = \Q_2(\sqrt[3]{2},\nu): \nu^2 + \nu + 1 = 0\\\hline
u_1 & -2\nu - 1 \\
u_2 &  (-2\nu - 1)\sqrt[3]{2}^2 + (2\nu + 1)\sqrt[3]{2} - 2\nu - 7\\
u_3 & \sqrt[3]{2}^2 - 2\nu - 1\\
u_4 & (-2\nu - 1)\sqrt[3]{2}^2 + (2\nu + 1)\sqrt[3]{2} + \nu - 1\\\midrule
\multicolumn{1}{c}{}& \Q_2(\sqrt[3]{2}, \alpha_{\pm})\\\hline
u_1 & \sqrt[3]{2}^2 + \sqrt[3]{2} + 1\\
u_2 & -\sqrt[3]{2}^2\alpha_{\pm} - 1\\
u_3 & (-2\sqrt[3]{2}^2 + 4)\alpha_{\pm} - \sqrt[3]{2}^2 - \sqrt[3]{2} + 3\\
u_4 & (-\sqrt[3]{2}^2 + 2)\alpha_{\pm} - \sqrt[3]{2}^2 - \sqrt[3]{2} + 1\\
u_5 & -\alpha_{\pm} + 1\\\midrule
\multicolumn{1}{c}{}& \Q_2(\sqrt[3]{2}, \beta_{\pm})\\\hline
u_1 & \sqrt[3]{2}^2 + 1\\
u_2 & -\beta_{\pm} - 1\\
u_3 & (-2\sqrt[3]{2} + 2)\beta_{\pm} - \sqrt[3]{2}^2 + 1\\
u_4 & (-\sqrt[3]{2} + 1)\beta_{\pm} + \sqrt[3]{2}^2 - \sqrt[3]{2} + 1\\
u_5 & (-\sqrt[3]{2} + 1)\beta_{\pm} - 2\sqrt[3]{2} + 1\\

\end{tabular}
\caption{\rule{0pt}{2.5ex}Generators for the group $(\calO_K/\frakf)^\times/U_{\frakf}$ for the quadratic extensions $K\,|\,F_2$ contained in $K_{1, r}$ and $K_{2,r}$, with $r = \pm1,\pm2$}
\label{table:cubic-gen-unit-grps-pr-2:ex}
\end{table}

\begin{table}
\normalfont\small\renewcommand{\arraystretch}{1.25}
\begin{tabular}{ >{$}c<{$} || >{$}c<{$} | >{$}c<{$} || >{$}r<{$} >{$}l<{$} }
K\,|\,F_2 & \mathfrak{f} & f & \multicolumn{2}{c}{$(\calO_K/\frakf)^\times/U_{\frakf}$} \\\hline\hline
\multirow{8}{*}{$\Q_4(\sqrt[3]{2})$} & \multirow{8}{*}{$\frakp^f$} & 1 &  \langle u_4 \rangle \!\!\!\!\!&\simeq\Z/3 \\
&   & 2 & \langle u_4 \rangle \!\!\!\!\!&\simeq  \Z/6 \\
&   & 3 & \langle u_4 \rangle \!\!\!\!\!&\simeq \Z/12 \\
&   & 4 & \langle u_3 \rangle \times \langle u_4 \rangle \!\!\!\!\!&\simeq \Z/2 \times \Z/12 \\
&   & 5 & \langle u_3 \rangle \times \langle u_4 \rangle \!\!\!\!\!&\simeq \Z/2 \times \Z/24 \\
&   &  6 & \langle u_2 \rangle \times \langle u_3 \rangle\times \langle u_4 \rangle \!\!\!\!\!&\simeq \Z/2 \times \Z/2 \times \Z/24 \\
&   &  \ge 7 & \langle u_1 \rangle \times \langle u_2 \rangle\times \langle u_3 \rangle \times \langle u_4 \rangle\!\!\!\!\!&\simeq\\
& & &\multicolumn{2}{c}{\hfill $\Z/2 \times \Z/2^{\lfloor \frac{f - 7}{3} \rfloor + 1} \times \Z/2^{\lfloor \frac{f - 6}{3}\rfloor + 1} \times \Z/(3\cdot 2^{\lfloor\frac{f - 5}{3}\rfloor + 3})$}\\\hline
\multirow{12}{*}{$\Q_2(\sqrt[3]{2}, \alpha_{\pm})$} & \multirow{12}{*}{$\frakp^f$} & 1 & \multicolumn{2}{c}{trivial} \\
&   & 2 & \langle u_5 \rangle \!\!\!\!\!&\simeq  \Z/2 \\
&   & 3 & \langle u_5 \rangle \!\!\!\!\!&\simeq \Z/4 \\
&   & 4, 5 & \langle u_4 \rangle \times \langle u_5 \rangle \!\!\!\!\!&\simeq \Z/2 \times \Z/4 \\
&   &  6, 7 & \langle u_3 \rangle \times \langle u_4 \rangle \times \langle u_5 \rangle \!\!\!\!\!&\simeq \Z/2 \times \Z/2 \times \Z/4 \\
&   &  8, 9 & \langle u_3 \rangle \times \langle u_4 \rangle \times \langle u_5 \rangle \!\!\!\!\!&\simeq \Z/2 \times \Z/2 \times \Z/8 \\
&   &  10, 11 &\langle u_2 \rangle \times \langle u_3 \rangle \times \langle u_4 \rangle \times \langle u_5 \rangle \!\!\!\!\!&\simeq \Z/2 \times \Z/2 \times \Z/2 \times \Z/8 \\
&   &  12, 13 & \langle u_1 \rangle \times \langle u_2 \rangle \times \langle u_3 \rangle \times \langle u_4 \rangle \times \langle u_5 \rangle \!\!\!\!\!&\simeq \Z/2 \times \Z/2 \times \Z/2 \times \Z/2 \times \Z/8 \\
&   &  14, 15 &  \langle u_1 \rangle \times \langle u_2 \rangle \times \langle u_3 \rangle \times \langle u_4 \rangle \times \langle u_5 \rangle \!\!\!\!\!&\simeq \Z/2 \times \Z/2 \times \Z/2 \times \Z/2 \times \Z/16 \\
&   &  16, 17 &  \langle u_1 \rangle \times \langle u_2 \rangle \times \langle u_3 \rangle \times \langle u_4 \rangle \times \langle u_5 \rangle \!\!\!\!\!&\simeq \Z/2 \times \Z/2 \times \Z/2 \times \Z/4 \times \Z/16 \\
&   &  \ge 18 & \langle u_1 \rangle \times \langle u_2 \rangle\times \langle u_3 \rangle \times \langle u_4 \rangle \times \langle u_5 \rangle \!\!\!\!\!&\simeq \\
& & & \multicolumn{2}{c}{\hfill $\Z/2 \times \Z/2 \times \Z/2^{\lfloor \frac{f - 18}{6} \rfloor + 2} \times \Z/2^{\lfloor \frac{f - 16}{6}\rfloor + 2} \times \Z/2^{\lfloor\frac{f - 14}{6}\rfloor + 4}$}\\
\hline
\multirow{13}{*}{$\Q_2(\sqrt[3]{2}, \beta_{\pm})$} & \multirow{13}{*}{$\frakp^f$} & 1 & \multicolumn{2}{c}{trivial} \\
&   & 2, 3 & \langle u_5 \rangle \!\!\!\!\!&\simeq  \Z/2 \\
&   & 4, 5 & \langle u_4 \rangle \times \langle u_5 \rangle \!\!\!\!\!&\simeq \Z/2 \times \Z/2 \\
&   &  6, 7 &\langle u_3 \rangle \times \langle u_4 \rangle\times \langle u_5 \rangle \!\!\!\!\!&\simeq \Z/2 \times \Z/2 \times \Z/2 \\
&   &  8, 9 &\langle u_2 \rangle \times \langle u_3 \rangle\times \langle u_4 \rangle \times \langle u_5 \rangle \!\!\!\!\!&\simeq \Z/2 \times \Z/2 \times \Z/2 \times \Z/2 \\
&   &  10 &\langle u_2 \rangle \times \langle u_3 \rangle\times \langle u_4 \rangle \times \langle u_5 \rangle \!\!\!\!\!&\simeq \Z/2 \times \Z/2 \times \Z/2 \times \Z/4 \\
&   &  11 &\langle u_2 \rangle \times \langle u_3 \rangle\times \langle u_4 \rangle \times \langle u_5 \rangle \!\!\!\!\!&\simeq \Z/2 \times \Z/2 \times \Z/4 \times \Z/4 \\
&   &  12, 13 & \langle u_1 \rangle \times \langle u_2 \rangle\times \langle u_3 \rangle \times \langle u_4 \rangle \times \langle u_5 \rangle \!\!\!\!\!&\simeq \Z/2 \times \Z/2 \times \Z/2 \times \Z/4 \times \Z/4 \\
&   &  14, 15 &\langle u_1 \rangle \times \langle u_2 \rangle\times \langle u_3 \rangle \times \langle u_4 \rangle \times \langle u_5\rangle \!\!\!\!\!&\simeq \Z/2 \times \Z/2 \times \Z/4 \times \Z/4 \times \Z/4 \\
&   &  16, 17 &\langle u_1 \rangle \times \langle u_2 \rangle\times \langle u_3 \rangle \times \langle u_4 \rangle \times \langle u_5\rangle \!\!\!\!\!&\simeq \Z/2 \times \Z/2 \times \Z/2 \times \Z/4 \times \Z/8 \\
&   &  18, 19 &\langle u_1 \rangle \times \langle u_2 \rangle\times \langle u_3 \rangle \times \langle u_4 \rangle \times \langle u_5\rangle \!\!\!\!\!&\simeq \Z/2 \times \Z/2 \times \Z/4 \times \Z/8 \times \Z/8\\
&   &  \ge 20 & \langle u_1 \rangle \times \langle u_2 \rangle\times \langle u_3 \rangle \times \langle u_4 \rangle \times \langle u_5 \rangle \!\!\!\!\!&\simeq\\
& & & \multicolumn{2}{c}{\hfill $\Z/2 \times \Z/2 \times \Z/2^{\lfloor \frac{f - 20}{6} \rfloor + 3} \times \Z/2^{\lfloor \frac{f - 18}{6}\rfloor + 3} \times \Z/2^{\lfloor\frac{f - 16}{6}\rfloor + 3}$}\\

\end{tabular}
\caption{\rule{0pt}{2.5ex}Group structure of $(\calO_{\Mquad}/\frakf)^\times/U_{\frakf}$ for the quadratic extensions $K\,|\,F_2$ contained in $K_{1, r}$ and $K_{2,r}$, with $r = \pm1,\pm2$}
\label{table:U-quo-unit-grps-pr-2:ex}
\end{table}

\begin{table}
\normalfont\small\renewcommand{\arraystretch}{1.25}
\begin{tabular}{ >{$}c<{$} || >{$}c<{$}| >{$}c<{$} || >{$}c<{$} | >{$}c<{$} ||  >{$}c<{$} | >{$}c<{$} }
{\Mquad} & \mathfrak{f} & f & r & \textup{values of~$\chi$ on generators} & \tau|_{F_2} & \condexp(\tau|_{F_2}) \\\hline\hline
\multirow{4}{*}{$\Q_2(\sqrt[3]{2},\alpha_{+})$} & \multirow{4}{*}{$\frakp^f$} & 3 & 4 & i &\tauSCnew{\frakp}{\Q_2(\sqrt[3]{2},\alpha_{+}),3,4} & 5\\
 &  & 6 & 4 & -1,\,1,\, i & \tauSCnew{\frakp}{\Q_2(\sqrt[3]{2},\alpha_{+}),6,4} & 8\\
 &  & 12 & 4 & -1, -1,\,1,\, 1,\, i& \tauSCnew{\frakp}{\Q_2(\sqrt[3]{2},\alpha_{+}),12,4} & 14\\ 
 &  & 12 & 4 & -1,\,1,\,1,\, -1,\, i & \tauSCnew{\frakp}{\Q_2(\sqrt[3]{2},\alpha_{+}),12,4} & 14\\\hline
\multirow{4}{*}{$\Q_2(\sqrt[3]{2},\beta_{+})$} & \multirow{4}{*}{$\frakp^f$} & 11 & 4 & -1,\,1,\, -i,\,i & \tauSCnew{\frakp}{\Q_2(\sqrt[3]{2},\beta_{+}),11,4} & 17\\
 &  & 11 & 4 & 1,\,1,\, i,\,-i & \tauSCnew{\frakp}{\Q_2(\sqrt[3]{2},\beta_{+}),11,4} & 17\\ 
 &  & 11 & 4 & 1,\,-1,\, -i,\,i & \tauSCnew{\frakp}{\Q_2(\sqrt[3]{2},\beta_{+}),11,4} & 17\\
 &  & 11 & 4 & -1,\, -1,\, -i,\,i& \tauSCnew{\frakp}{\Q_2(\sqrt[3]{2},\beta_{+}),11,4} & 17\\
 
\end{tabular}
\caption{\rule{0pt}{2.5ex}Restrictions of primitive inertial types from $\Q_2$ to $F_2$}
\label{table:type-cubic_root_2}
\end{table}

For the rest of this section, we write $F_2 \colonequals \Q_2(\sqrt[3]{2})$.

\begin{proposition}\label{prop:imprimitive-triple} Let $\tau$ be an exceptional type arising from $E$ over $\Q_2$. Then $\tau|_{I_{F_2}}$ is one of the (non-exceptional) supercuspidal types listed in
Table~\textup{\ref{table:type-cubic_root_2}}.
\end{proposition}

\begin{proof}  Let $\tau := \rho_E |_{I_2}$ be an exceptional type arising from $E$ over $\Q_2$. From Theorem~\ref{thm:exceptell2} and its proof we know $\tau$ corresponds to a field $L: = K_{i,r}$ for some $i,r$.
Since $F_2$ and its two conjugated extensions
are the unique cubic extension of $\Q_2$ inside all the $K_{i,r}$, we conclude that
$\tau|_{I_{F_2}}$ is imprimitive by Proposition~\ref{prop:triply-imprimitive:2}. Furthermore, $\tau|_{I_{F_2}}$ has conductor exponent $5$, $8$, $14$, or $17$.
Let $K\,|\,F_2$ be the quadratic extension from which $\tau|_{I_{F_2}}$ is induced. Let $\chi\colon W_K \to \C^\times$ be the character such that
\[ \tau|_{I_{F_2}} = \rho_E|_{I_{F_2}} \]
where
\[ \rho_E|_{G_{F_2}} \colonequals \Ind_{K}^{F_2}\chi. \]
By the conductor exponent formula~\eqref{E:condBC}, $\chi$ has conductor exponent $3$, $6$, $11$, or $12$. We also know that
$$\ker \chi\spArt = \Nm_{L|K}(\calO_L^\times)/U_{\frakf}\hookrightarrow (\calO_K/\frakf)^\times/U_{\frakf}.$$
In Table~\ref{table:kernel-Q2:ex}, we compute the norm group $\Nm_{L|K}(\calO_L^\times)/U_{\frakf}$ for each possible $L$ and each of the ramified
quadratic extensions $K\,|\,F_2$ listed in Lemma~\ref{lem:subfields} (code at \cite{ourcode}). Since the projective image of $\mathrm{P}(\rho_E|_{G_{F_2}}))$ is isomorphic to $D_4$, we conclude that $\chi$ is of order~4 on $I_K$. Now, using all the previous constraints we determine the \textit{unique} quadratic extension $K\,|\,F_2$ from which $\tau|_{I_{F_2}}$ is induced, together with the corresponding character $\chi$.
Indeed, the order of the norm groups suffices to decide what is the correct quadratic extension in all cases.
For the correct quadratic extension $K\,|\,F_2$, Table~\ref{table:kernel-Q2:ex} also gives the generators of $\ker \chi\spArt$ in terms of those of the group
$(\calO_K/\frakf)^\times/U_{\frakf}$ listed in Table~\ref{table:cubic-gen-unit-grps-pr-2:ex}. 
Table~\ref{table:type-cubic_root_2} contains, up to complex conjugation, all the inertial types satisfying those conditions. Therefore, $\tau|_{I_{F_2}}$ must be one the types listed in it.
\end{proof}
 
\begin{table}
\normalfont\small\renewcommand{\arraystretch}{1.25}
\begin{tabular}{ >{$}c<{$} || >{$}c<{$} | >{$}c<{$} | >{$}c<{$}  >{$}c<{$} |  >{$}c<{$}}
E & f & K\,|\,F_2 & \multicolumn{2}{c|}{$\Nm_{L|K}(\calO_L^\times)/U_\frakf \leq (\calO_K/\frakf)^\times/U_\frakf$} & \condexp(\tau_E|_{F_2}) \\
\hline\hline
\multirow{2}{*}{$E_{1, 1}$} &  \multirow{2}{*}{$3$}  & \Q_2(\sqrt[3]{2},\alpha_{-}) & \multicolumn{2}{c|}{$\Z/2$} &  \multirow{2}{*}{$5$}\\
& &  \Q_2(\sqrt[3]{2},\alpha_{+}) & \multicolumn{2}{c|}{\textup{trivial}} & \\
\hline
\multirow{2}{*}{$E_{1,-1}$} &  \multirow{2}{*}{$6$} & \Q_2(\sqrt[3]{2},\alpha_{-}) & \multicolumn{2}{c|}{$\Z/2 \times \Z/2 \times \Z/2$} &  \multirow{2}{*}{$8$}\\
& &  \Q_2(\sqrt[3]{2},\alpha_{+}) &  \multicolumn{2}{c|}{$\left\langle u_3 u_4 u_5^2, u_4 \right\rangle\simeq \Z/2 \times \Z/2$} & \\
\hline
\multirow{3}{*}{$E_{1, 2}$} &  \multirow{3}{*}{$12$}  & \Q_2(\sqrt[3]{2},\alpha_{-})& \multicolumn{2}{c|}{$\Z/2 \times \Z/2 \times \Z/2 \times \Z/2 \times \Z/4$} & \multirow{3}{*}{$14$}\\
& & \multirow{2}{*}{$\Q_2(\sqrt[3]{2},\alpha_{+})$} & \left\langle u_1u_2u_4, u_1u_2 u_5^4, u_1u_2 u_3 u_5^4, u_2 u_5^6\right\rangle \!\!\!\!\!&\phantom{cccccccccc}\\ 
& & & \multicolumn{2}{c|}{\hfill $\simeq \Z/2 \times \Z/2 \times \Z/2 \times \Z/4$} & \\\hline
\multirow{3}{*}{$E_{1,-2}$} &  \multirow{3}{*}{$12$} & \Q_2(\sqrt[3]{2},\alpha_{-})  &  \multicolumn{2}{c|}{$\Z/2 \times \Z/2 \times \Z/2 \times \Z/2 \times \Z/4$} & \multirow{3}{*}{$14$}\\
& & \multirow{2}{*}{$\Q_2(\sqrt[3]{2},\alpha_{+})$} &  \left\langle u_1u_2u_4, u_2 u_5^4, u_2u_3u_5^4, u_1u_3 u_5^2\right\rangle \!\!\!\!\!&\phantom{cccccccccc}\\ 
& & & \multicolumn{2}{c|}{\hfill $\simeq\Z/2 \times \Z/2 \times \Z/2 \times \Z/4$} & \\
\midrule
\multirow{2}{*}{$E_{2, 1}$} &  \multirow{2}{*}{$11$}  & \Q_2(\sqrt[3]{2},\beta_{-}) & 
\multicolumn{2}{c|}{$\left\langle u_3 u_4^2u_5^2, u_2 u_4^2, u_3 u_4 u_5^3\right\rangle \simeq\Z/2 \times \Z/2 \times \Z/4$} & \multirow{2}{*}{$17$}\\ 
& & \Q_2(\sqrt[3]{2},\beta_{+}) &  \multicolumn{2}{c|}{$\Z/2 \times \Z/2 \times \Z/2 \times \Z/4$} & \\
\hline
\multirow{2}{*}{$E_{2, -1}$} &  \multirow{2}{*}{$11$}  & \Q_2(\sqrt[3]{2},\beta_{-}) & 
\multicolumn{2}{c|}{$\left\langle u_2, u_3 u_4^2 u_5^2, u_4 u_5\right\rangle \simeq\Z/2 \times \Z/2 \times \Z/4$}  & \multirow{2}{*}{$17$}\\
& & \Q_2(\sqrt[3]{2},\beta_{+}) &  \multicolumn{2}{c|}{$\Z/2 \times \Z/2 \times \Z/2 \times \Z/4$} & \\
\hline
\multirow{2}{*}{$E_{2, 2}$} &  \multirow{2}{*}{$11$}  &  \Q_2(\sqrt[3]{2},\beta_{-}) &  
\multicolumn{2}{c|}{$\left\langle u_3u_5^2, u_2, u_3 u_4 u_5^3\right\rangle \simeq\Z/2 \times \Z/2 \times \Z/4$}   & \multirow{2}{*}{$17$}\\
& & \Q_2(\sqrt[3]{2},\beta_{+}) &  \multicolumn{2}{c|}{$\Z/2 \times \Z/2 \times \Z/2 \times \Z/4$} & \\
\hline
\multirow{2}{*}{$E_{2,-2}$} &  \multirow{2}{*}{$11$} & \Q_2(\sqrt[3]{2},\beta_{-}) & 
\multicolumn{2}{c|}{$\left\langle u_3u_5^2, u_2 u_4^2, u_2 u_3u_4 u_5^3\right\rangle \simeq\Z/2 \times \Z/2 \times \Z/4$}  & \multirow{2}{*}{$17$}\\
& & \Q_2(\sqrt[3]{2},\beta_{+}) &  \multicolumn{2}{c|}{$\Z/2 \times \Z/2 \times \Z/2 \times \Z/4$} & \\

\end{tabular}
\caption{\rule{0pt}{2.5ex} Norm groups $\Nm_{L|K}(\calO_L^\times)/U_{\frakf}$ and kernels $\ker \chi\spArt$ of the inducing characters for the restrictions of exceptional types to $F_2$.}
\label{table:kernel-Q2:ex}
\end{table}

\subsection{Explicit realization} \label{sec:exm:2ex}

In Table~\ref{table:types-Q2:ex} we give all exceptional inertial types for elliptic curves over~$\Q_2$ with potentially good reduction together with a curve realizing each type.
In particular, this shows that all the possible types over~$\Q_2$ we compute indeed arise from elliptic curves.  A similar corollary holds as Corollary~\ref{cor:algQ3}.

\begin{table}
\normalfont\small\renewcommand{\arraystretch}{1.25}
\begin{tabular}{ >{$}c<{$} || >{$}c<{$} | >{$}c<{$} | >{$}c<{$} | >{$}c<{$} || >{$}c<{$} }
\tau & e & \condexp(\tau) & L' & \Gal(L\,|\,\Q_2) & E \\
\hline\hline
\tauex{1} &  24 & 3 & \LMFDBL{2.8.10.2} & \LMFDBG{8T23} \simeq \GL_2(\F_3) & \LMFDBE{648b1} \\
\hline
\tauex{1} \otimes \epschar{-4} &  24 & 4 & \LMFDBL{2.8.12.29} & \LMFDBG{8T23} \simeq \GL_2(\F_3)& \LMFDBE{1296c1} \\
\hline
\tauex{1} \otimes \epschar{8} &  24 & 6 & \LMFDBL{2.8.16.73} & \LMFDBG{8T23} \simeq \GL_2(\F_3)& \LMFDBE{5184e1} \\
\hline
\tauex{1} \otimes \epschar{-8} &  24 & 6 & \LMFDBL{2.8.16.71} & \LMFDBG{8T23} \simeq \GL_2(\F_3)& \LMFDBE{5184w1} \\
\hline
\tauex{2} &  24 & 7 & \LMFDBL{2.8.22.132} & \LMFDBG{8T23} \simeq \GL_2(\F_3)&\LMFDBE{3456a1} \\
\hline
\tauex{2} \otimes \epschar{-4} & 24 & 7 & \LMFDBL{2.8.22.136} & \LMFDBG{8T23} \simeq \GL_2(\F_3)& \LMFDBE{3456e1}\\
\hline
\tauex{2} \otimes \epschar{8} &  24 & 7 & \LMFDBL{2.8.22.138} & \LMFDBG{8T23} \simeq \GL_2(\F_3)&\LMFDBE{3456o1} \\
\hline
\tauex{2} \otimes \epschar{-8} &  24 & 7 & \LMFDBL{2.8.22.137} & \LMFDBG{8T23} \simeq \GL_2(\F_3)&\LMFDBE{3456c1} \\
\end{tabular}
\caption{\rule{0pt}{2.5ex}Elliptic curves for each exceptional inertial type over $\Q_2$.}
\label{table:types-Q2:ex}
\end{table}

\bibliographystyle{plain}

\end{document}